\documentclass[11pt, twoside, leqno]{article}
\usepackage{enumerate}
\usepackage{graphics}
\usepackage{graphicx}
\usepackage{amssymb}
\usepackage{amsmath}
\usepackage{amsthm}
\usepackage{color}
\usepackage{mathrsfs}

\usepackage{indentfirst}

\usepackage{txfonts}

\allowdisplaybreaks

\def\red{\color{red}}

\def\rn{{\mathbb{R}^n}}

\def\zz{{\mathbb Z}}

\def\nn{{\mathbb N}}

\def\cx{{\mathcal X}}

\def\fz{\infty }

\def\lf{\left}
\def\r{\right}

\def\ls{\lesssim}

\def\noz{\nonumber}

\def\supp{\mathop\mathrm{\,supp\,}}

\def\Xint#1{\mathchoice
	{\XXint\displaystyle\textstyle{#1}}%
	{\XXint\textstyle\scriptstyle{#1}}%
	{\XXint\scriptstyle\scriptscriptstyle{#1}}%
	{\XXint\scriptscriptstyle\scriptscriptstyle{#1}}%
	\!\int}
\def\XXint#1#2#3{{\setbox0=\hbox{$#1{#2#3}{\int}$ }
		\vcenter{\hbox{$#2#3$ }}\kern-.6\wd0}}

\def\ta{\theta}

\def\f{\frac}
\def\vi{\varphi}

\def\({\left(}
\def \){ \right)}

\def\da{\delta}

\def\t{{\theta}}

\def\va{\varepsilon}

\def\BB{{\mathbb B}}

\newcommand{\wt}{\widetilde}

\newcommand{\p}{\partial}

\def\al{\alpha}

\def\RR{\mathbb{R}}

 \def\supp{\operatorname{supp}}

\def\BB{\mathbb{B}}

\def\ld{\lambda}

\def\og{\omega}

\newcommand{\bp}{ \begin{proof} }
	
	\newcommand{\ep}{ \end{proof} }

\def\Xint#1{\mathchoice
	{\XXint\displaystyle\textstyle{#1}}%
	{\XXint\textstyle\scriptstyle{#1}}%
	{\XXint\scriptstyle\scriptscriptstyle{#1}}%
	{\XXint\scriptscriptstyle\scriptscriptstyle{#1}}%
	\!\int}
\def\XXint#1#2#3{{\setbox0=\hbox{$#1{#2#3}{\int}$ }
		\vcenter{\hbox{$#2#3$ }}\kern-.6\wd0}}

\def\av{\Xint-}

\def\NN{{\mathbb N}}

\newtheorem{theorem}{Theorem}[section]
\newtheorem{lemma}[theorem]{Lemma}
\newtheorem{corollary}[theorem]{Corollary}

\newtheorem{proposition}[theorem]{Proposition}

\theoremstyle{definition}
\newtheorem{remark}[theorem]{Remark}
\newtheorem{definition}[theorem]{Definition}
\renewcommand{\appendix}{\par
	\setcounter{section}{0}%
	\setcounter{subsection}{0}%
	\setcounter{subsubsection}{0}%
	\gdef\thesection{\@Alph\c@section}%
	\gdef\thesubsection{\@Alph\c@section.\@arabic\c@subsection}%
	\gdef\theHsection{\@Alph\c@section.}%
	\gdef\theHsubsection{\@Alph\c@section.\@arabic\c@subsection}%
	\csname appendixmore\endcsname
}

\numberwithin{equation}{section}
\def\cX{\mathcal{X}}
\def\Lip{\operatorname{LIP}}
\def\lip {\operatorname{lip}}
\def\blip {\operatorname{Lip}}
\def\cC{\mathcal{C}}

\begin{document}
	
\title{\bf\Large Poincar\'e
Inequality Meets Brezis--Van Schaftingen--Yung Formula on Metric Measure Spaces
	\footnotetext{\hspace{-0.35cm} 2020 {\it
			Mathematics Subject Classification}. Primary 46E36;
		Secondary 26D10, 30L15, 26A33.
		\endgraf {\it Key words and phrases.} metric
		measure space, Poincar\'e inequality, Brezis--Van
        Schaftingen--Yung Formula, Sobolev semi-norm,
		fractional Sobolev inequality, Gagliardo--Nirenberg
		inequality.
		\endgraf The first author is supported by
		NSERC of Canada Discovery grant RGPIN-2020-03909.
		This project is also supported by the National
		Natural Science Foundation of China (Grant Nos.
		11971058, 12071197, 12122102 and 11871100)
		and the National Key Research
		and Development Program of China
		(Grant No. 2020YFA0712900).}}
\author{Feng Dai, Xiaosheng Lin, Dachun Yang\footnote{Corresponding
author, E-mail: \texttt{dcyang@bnu.edu.cn}/{\red November 18, 2021}/Final version.},
	\ Wen Yuan and Yangyang Zhang}
\date{}
\maketitle

\vspace{-0.8cm}

\begin{center}
\begin{minipage}{13cm}
{\small {\bf Abstract}\quad
Let  $(\mathcal{X}, \rho, \mu)$ be a  metric
measure space of homogeneous type which
supports a certain Poincar\'e inequality.
Denote by the symbol  $\mathcal{C}_{\mathrm{c}}^\ast(\mathcal{X})$
the space  of  all  continuous functions $f$ with compact support
satisfying that $\operatorname{Lip} f:=\limsup _{r \rightarrow 0}
\sup_{y\in B(\cdot, r)} |f(\cdot)-f(y)|/r$ is also a continuous function
with compact support and
$\operatorname{Lip} f=\lim _{r \rightarrow 0} \sup_{y\in B(\cdot, r)}
|f(\cdot)-f(y)|/r$ converges uniformly.
Let $p \in[1,\infty)$. In this article, the authors prove that,
for any $f\in\mathcal{C}_{\mathrm{c}}^\ast(\mathcal{X})$,
\begin{align*}
&\sup_{\lambda\in(0,\infty)}\lambda^p\int_{\mathcal{X}}
\mu\left(\left\{y\in \mathcal{X}:\ |f(x)-f(y)|>\lambda \rho(x,y)
[V(x,y)]^{\frac 1p}\right\}\right)\, d\mu(x)\\
&\quad\sim \int_{{\mathcal{X}}} [\operatorname{Lip}f(x)]^p  \,d\mu(x)	
\end{align*}
with the positive equivalence constants independent of $f$,
where $V(x,y):=\mu(B(x,\rho(x,y)))$.
This generalizes a recent surprising formula of H. Brezis,
J. Van Schaftingen, and P.-L. Yung from the $n$-dimensional 
Euclidean space ${\mathbb R}^n$
to $\mathcal{X}$. Applying this generalization, the
authors establish new fractional Sobolev and
Gagliardo--Nirenberg inequalities in $\mathcal{X}$.
All these results have
a wide
range of applications. Particularly,
when  applied to
two concrete examples,
namely, $\rn$ with weighted Lebesgue measure
and the complete $n$-dimensional Riemannian manifold with non-negative
Ricci curvature, all these results are new.
The proofs of these results strongly depend on
the geometrical
relation of differences and derivatives in the metric measure space and
the Poincar\'e inequality.}
\end{minipage}
\end{center}
	
\vspace{0.2cm}

\section{Introduction}

For any given $s\in(0,1)$ and $p\in[1,\infty)$, the \emph{homogeneous fractional
Sobolev space} $\dot{W}^{s,p}(\rn)$ consists of all measurable
functions $f$ on $\rn$ having the finite Gagliardo semi-norm
\begin{align}\label{wsp}
\|f\|_{\dot{W}^{s,p}(\rn)}:=\left[\int_{\rn}\int_{\rn}\frac{|f(x)-f(y)|^p}
{|x-y|^{n+sp}}\,dx\,dy\right]^{\frac{1}{p}}
=:\left\| \f {f(x)-f(y)} {|x-y|^{\f np+s}}\right\|_{L^p(\rn\times\rn)}.
\end{align}
These spaces play a key role in harmonic analysis
and partial differential equations (see, for instance, \cite{bbm02,m11,npv,crs10,cv11,mm11}).

A well-known \emph{drawback} of the Gagliardo semi-norm
in \eqref{wsp} is that one does not
recover the homogeneous Sobolev semi-norm $\|\nabla f\|_{L^p(\rn)}$ when $s=1$,
in which case the integral in \eqref{wsp} is infinite unless $f$
is a constant (see \cite{bbm,ref7,BSY212}), here and thereafter, for any
differentiable function $f$ on $\rn$, $\nabla f$ denotes the
gradient of $f$, namely, 
$$\nabla f:=\lf(\frac{\partial f}{\partial x_1},
\ldots,\frac{\partial f}{\partial x_n}\r).$$
An important approach to recover $\|\nabla f \|_{L^p(\rn)}$ out of
the Gagliardo semi-norm is due to Bourgain et al.
\cite{bbm01}, who in particular proved that
\begin{align*}
\lim_{s\in(0,1), s\to1}(1-s)\|f\|_{\dot{W}^{s,p}(\rn)}^p=C_{(p,n)}
 \|\nabla f\|_{L^p(\rn)}^p, \ \ \ \forall\, f\in \dot{W}^{1,p}(\rn),\ \forall\,p\in[1,\infty),
\end{align*}
where $C_{(p,n)}$ is a positive constant depending only on $p$ and $n.$
Very recently, Brezis et al.  \cite{ref8} (see also \cite{BSSY21})
discovered an alternative way to repair this \emph{defect}
by replacing the $L^p$ norm in \eqref{wsp} with the weak
$L^p$ quasi-norm [namely, $\|\cdot\|_{L^{p,\infty}(\rn\times\rn)}$].
For any given $p\in[1,\infty),$ Brezis et al. in
\cite{ref8} proved that there exist positive constants
$C_1$ and $C_2$ such that, for any
 $f\in C_{\mathrm{c}}^\infty(\rn)$,
\begin{align}\label{bsy}
C_1 \|\nabla f\|_{L^p(\rn)}\leq \left\|\frac{f(x)-f(y)}{|x-y|
^{\frac{n}{p}+1}}\right\|
_{L^{p,\infty}(\rn\times\rn)}\le C_2 \|\nabla f\|_{L^p(\rn)},
\end{align}
where
\begin{align*}
\left\|\frac{f(x)-f(y)}{|x-y|^{\frac{n}{p}+1}}\right\|_{L^{p,\infty}
(\rn\times\rn)}:&=\sup_{\ld\in(0,\infty)}\lambda
\left|\left\{(x,y)\in\rn\times\rn:\
\frac{|f(x)-f(y)|}{|x-y|^{\frac{n}{p}+1}}>
\lambda\right\}\right|^{\frac{1}{p}},
\end{align*}
here and thereafter, the symbol
$|E|$ denotes the Lebesgue
measure of a measurable set $E\subset \RR^m$ for any given $m\in\nn$,
and the symbol $C_{\mathrm{c}}^\fz(\rn)$ the
set of all infinitely differentiable functions on $\rn$
with compact support. The equivalence \eqref{bsy} in
particular allows Brezis et al. in \cite{ref8} to derive
some surprising alternative estimates of fractional Sobolev and
Gagliardo--Nirenberg
inequalities in some exceptional
cases involving $\dot{W}^{1,1}(\rn)$, where
the \emph{anticipated} fractional Sobolev and
Gagliardo--Nirenberg inequalities fail. For
more studies on fractional
Sobolev spaces, we refer the reader to \cite{GY21,DLYYZ21,BSSY21,BSY212, BSYar,Par}.

Let us also give a few comments on the proof of
\eqref{bsy} in \cite{ref8}. Indeed, a
substantially sharper lower bound was obtained
in \cite{ref8}, using a method of rotation and
the Taylor remainder theorem. On the other hand,
the stated
upper bound for any given $p\in[1,\infty)$ can be
deduced from
the Vitali covering lemma in one variable,  and a
method of rotation. Thus, the rotation invariance
of the space $L^p(\rn)$ seems to play a vital
role in the proof of
\eqref{bsy} in \cite{ref8}.

The main purpose in this article is to give an essential
extension of the main results of \cite{ref8}
[the equivalence \eqref{bsy}]. Such
extensions are fairly nontrivial because our
setting is a given metric measure  space of homogeneous type, in
the sense of Coifman and Weiss \cite{CW71,CW77}, that
is neither rotation invariance nor translation invariance.
We now describe this general setting. First, we recall the notion of  metric
measure spaces of homogeneous type.  For
more studies on  spaces of homogeneous type, we refer the reader to
\cite{BD20,BDL,L10,LW13,N06,NY97,HHL16,HHL17}. A
\emph{metric space} $(\mathcal{X},\rho)$ is a
non-empty
set $\mathcal{X}$ equipped with a \emph{metric} $\rho,$
namely, a non-negative function, defined on $\mathcal{X}
\times\mathcal{X},$ satisfying that, for any $x,$ $y,$ $z\in\mathcal{X},$
\begin{enumerate}[(i)]
	\item $\rho(x,y)=0$ if and only if $x=y;$
	\item $\rho(x,y)=\rho(y,x);$
	\item $\rho(x,z)\leq \rho(x,y)+\rho(y,z).$
\end{enumerate}
For any $x\in\mathcal{X}$ and $r\in(0,\infty),$ let
$$B(x,r):=\{y\in\mathcal{X}:\
\rho(x,y)<r\}$$ and
\begin{align*}
	\mathbb{B}:=\{B(x,r):\
	x\in\mathcal{X}\ \mathrm{and}\ r\in(0,\infty)\}.
\end{align*}
For any $\alpha\in(0,\infty)$ and any ball $B:=B(x_B,r_B)$ in $\mathcal{X}$, with $x_B\in\mathcal{X}$
and $r_B\in(0,\infty)$, let $\alpha B:=B(x_B,\alpha
r_B)$. For a given  metric space $(\mathcal{X},
\rho)$ and a given Borel measure $\mu,$ the
triple
$(\mathcal{X}, \rho, \mu)$ is called a \emph{metric measure space of
	homogeneous type}  if $\mu$ satisfies the following
\emph{doubling condition}: there exists a
constant $L_{\mu}\in[1,\infty)$ such that, for any ball
$B\subset \mathcal{X},$
\begin{align*}
	\mu(2B)\leq L_{\mu}\mu(B).
\end{align*}
From the above doubling condition, it follows that, for
any $\lambda \in[1,\infty)$ and any ball $B\subset\mathcal{X}$,
\begin{align}\label{1133}
\mu(\lambda B)\leq L_{\mu}\lambda
^{\log_2L_{\mu}}\mu(B).
\end{align}
Throughout this article, we always make the following \emph{assumptions}:
$\mu(B)\in(0,\infty)$ for any  $B\in\mathbb{B}$, and $\mu(\{x\})=0$ for any
$x\in\mathcal{X}.$

Now, we recall some notions to state the main result of this article.
A function $f:\cx\to \mathbb{C}$ on the metric space  $(\cx, \rho)$ is
called a \emph{Lipschitz function}
if there exists a positive constant $C$ such that,
 for any $x,\ y \in \cx$,
$$
|f(x)-f(y)| \leqslant C \rho(x, y).
$$
We denote by the \emph{symbol} $\operatorname{LIP}(\cx)$ the space of
all  Lipschitz functions on $\cx$.
For any  Lipschitz function $f$ defined on a metric space $(\cx, \rho)$
and  any $x \in \cx $, let
$$\operatorname{lip} f(x):=\liminf_{r\to  0} \sup _{y\in B(x, r)} \frac{|f(x)-f(y)|}{r}$$
and
$$\operatorname{Lip} f(x):=\limsup _{r \rightarrow 0} \sup_{y\in B(x, r)} \frac{|f(x)-f(y)|}{r}.$$
We denote by the \emph{symbol}
$BC({\mathcal{X}})$
the set of all bounded continuous functions on
${\mathcal{X}}$.
\begin{definition}\label{45622}
	Let $p,$ $q\in[1,\infty)$ and $(\cx, \rho, \mu)$ be a
	metric measure space of homogeneous type. A
	function $f\in\Lip(\cx)$ is said to
	satisfy a \emph{$(q,p)$-Poincar\'e inequality
	with constants} $C_1,$ $ C_2,$ $ \tau\in[1,\infty)$ if there
	exists a sequence $\{ \ell_B \}_{ B\in\BB}$
	of linear functionals on the space $BC(\cx)$
	such that, for any $\vi\in BC(\cx)$ and $B\in\BB$,
	\begin{equation}\label{3-9-a}
		\ell_B(1)=1,\ \ |\ell_B(\vi)|
		\leq C_1\left[ \f 1{\mu(B)} \int_{ B} |\vi (x)|^q
		\, d\mu(x)\right]^{\f1q},
	\end{equation}
	and
	\begin{equation}\label{2-8-eq}
		\left[\f 1 {\mu(B)}	\int_B |f(x)-\ell_B (f)|^q\,
		d\mu(x) \right]^{\f1q}\leq C_2 r_B \left[
		\f 1{\mu(\tau B)} \int_{\tau B} [\lip f(x)]^p
		\, d\mu(x)\right]^{\f1p}.
	\end{equation}
The space $(\cx, \rho, \mu)$ is said to satisfy a \emph{$(q,p)$-Poincar\'e inequality
	with constants} $C_1,$ $ C_2,$ $ \tau\in[1,\infty)$
if \eqref{3-9-a} and \eqref{2-8-eq} hold true
for any $f\in\Lip(\cx)$, and
the constants $C_1,$ $ C_2,$ and $\tau$ are independent of $f.$
\end{definition}

We denote by the \emph{symbol} $C_{\mathrm{c}}(\cX)$  the space of all
continuous functions with compact support  on $\cX$ and by the \emph{symbol}
$\cC_{\mathrm{c}}^\ast(\cX)$  the space  of all functions $f\in\Lip(\cX)$
such that $\blip f=\lip f\in C_{\mathrm{c}}(\cX) $  and
$$\operatorname{Lip} f(x)=\lim _{r \rightarrow 0} \sup_{y\in B(x, r)} \frac{|f(x)-f(y)|}{r}$$
uniformly  on  $x\in \cX$.
We use the \emph{symbol} $C_{\mathrm{c}}^2(\rn)$ to
denote the set of all twice continuously
differentiable functions on $\rn$ with compact support.
When $\cX:=\RR^n$ and $\rho$ is the Euclidean metric,
from the Taylor remainder theorem, it follows that $ C_{\mathrm{c}}^2(\RR^n)
\subset \cC_{\mathrm{c}}^\ast (\rn)$ (see Proposition \ref{1236} below).
For any $x,\ y\in\mathcal{X}$, let
$$ V (x, y) :=\mu\lf(B(x, \rho(x,y))
\r).$$
The main result of this article can  be stated as follows.
\begin{theorem}\label{thm-2-3}
	Let $p,\ q\in[1,\infty)$ and $(\cX, \rho,
	\mu)$ be a  metric
	measure space of homogeneous type. Assume that  $f\in\cC_{\mathrm{c}}^\ast (\cX)$
	satisfies a \emph{$(q,p)$-Poincar\'e
	inequality with constants} $C_1,$ $ C_2,$ $ \tau\in[1,\infty)$. Then
	\begin{equation}\label{2-9-eqq}
		\sup_{\ld\in(0,\infty)} 	\ld^p  \int_{{\mathcal{X}}}\int_{{\mathcal{X}}}
\mathbf{1}_{{D_\ld}}(x, y) \,d\mu(x) \,d\mu(y) \sim 	\int_{{\mathcal{X}}} [\blip f(x)]^p  \,d\mu(x),
	\end{equation}
	where
	\begin{align}\label{keydef}
		D_{\ld} :=\left\{ (x, y)\in {\mathcal{X}}\times {\mathcal{X}}:\
|f(x)-f(y)|>\ld \rho(x,y) [V(x,y)]^{\f1p} \right\}
	\end{align}
	and  the positive   equivalence constants depend only on $p$, $q,$ the
doubling constant of $\mu,$ and the  constants $C_1,$ $C_2,$ and $\tau$.
\end{theorem}

Using Theorem \ref{thm-2-3}, we are able to obtain some analogues
of fractional Sobolev and Gagliardo--Nirenberg
type inequalities for a general doubling measure $\mu$
on $\cx$ that satisfies certain Poincar\'e inequality (see Corollaries \ref{corollart6.5}
and \ref{corollary6.6} below), which correspond to the endpoint
cases of \cite[Corollaries 5.1 and 5.2]{ref8}.
\begin{remark}
\begin{itemize}
 	\item[$\mathrm{(i)}$] The formula \eqref{2-9-eqq} gives an equivalence
 		between the Sobolev semi-norm and a quantity involving the
 		difference of the function $f$. It is quite
 		remarkable that such an equivalence holds true
 		for the  metric
 		measure space of homogeneous type, $(\cX, \rho, \mu)$, that
 		admits a $p$-Poincar\'e inequality.
 		Indeed, finding an appropriate way to
 		characterize the smoothness of functions
 		via their finite differences is a
 		notoriously difficult problem in approximation
 		theory, even for some simple weights in
 		one variable (see \cite{K15, MT} and the
 		references therein). A major difficulty
 		comes from the fact that difference
 		operators $\Delta_h f:=f(\cdot +h)-f(\cdot)$
 		for any $h\in\rn$ are no longer bounded
 		on general weighted $L^p$ spaces.
 		
 		\item[$\mathrm{(ii)}$] In the special case of $\cx:=\rn,$ $d\mu(x)=dx$,
        and $\rho(x,y):=|x-y|$ for any $x,\ y\in\rn,$
 		Theorem \ref{thm-2-3} is due to Brezis et al. \cite{ref8}.
 		However, since proofs in \cite{ref8} largely
 		rely on the rotation invariance of the
 		Lebesgue measure on $\rn$, those proofs do not seem
 		to work in the general setting considered here.	
\end{itemize}
\end{remark}

Indeed,
the proof of the lower estimate in \eqref{2-9-eqq}
is
fairly nontrivial. On one hand, the proof of
the lower estimate in
 \eqref{bsy}
in the unweighted case in the article \cite{ref8} is
based on a method of polar coordinate transformation,
and it is impossible to use this method in the metric measure space.  Using
the geometrical relation of differences and  derivatives in
the metric measure space
(see Lemma \ref{lem-2-33} below),
we overcome this difficulty by accurately estimating the inclusion relation between
the set related to the difference and the set related to the derivative [see \eqref{2-5} below].
The proof of the upper estimate of \eqref{bsy}
in the article \cite{ref8} is
based on a method of rotation,
and  seems to be also inapplicable in the metric measure space here. To overcome these
difficulties, we first introduce a suitable form of the Poincar\'e inequality in
the metric measure space of homogeneous type (see Definition \ref{45622} above).
Applying the Poincar\'e inequality, then we skillfully use the Vitali covering
lemma in the product space [see \eqref{5-10b} below] to replace the
method of rotation in \cite{ref8},
which plays a key role in
the proof of the upper estimate of \eqref{2-9-eqq}.

The remainder of this article is organized as follows.

In Section \ref{section-3}, we establish the characterization of
the Sobolev semi-norm on  the  metric measure space of homogeneous type,
$(\cx, \rho, \mu)$ (see Theorem \ref{thm-2-3} above),
which can be easily deduced from two more general
results: Theorem \ref{thm-3-1} below and Theorem \ref{thm-2-2} below.
Indeed, by
the geometrical relation of differences and derivatives in
the metric measure space (see Lemma \ref{lem-2-33} below), we characterize the Sobolev semi-norm
in $\cx$, without
the assumption that $f\in\Lip(\cx)$ satisfies a
$(q,p)$-Poincar\'e inequality  in Theorem \ref{thm-2-3} (see Theorem \ref{thm-3-1} below).
To obtain the upper estimate of \eqref{2-9-eqq}, we need to overcome some
essential difficulties.  Using a skillful
decomposition, the Vitali covering lemma in the product space, and the Poincar\'e inequality,
we avoid the  rotation invariance used in
\cite{ref8} and hence overcome these difficulties (see Theorem \ref{thm-2-2} below).

Applying Theorem \ref{thm-2-3}, in Section \ref{section2}, we establish
fractional Sobolev and fractional Gagliardo--Nirenberg type inequalities on $\cx$
(see Corollaries \ref{corollart6.5} and \ref{corollary6.6} below).

In Section \ref{section3}, we apply Theorem \ref{thm-2-3} and
Corollaries \ref{corollart6.5} and \ref{corollary6.6}, respectively,
to two concrete examples of metric measure spaces,
namely, $\rn$ with weighted Lebesgue measure
and the complete $n$-dimensional Riemannian manifold with non-negative Ricci
curvature. All these results are totally new.

Finally, we make some conventions on notation.
Let $\nn:=\{1,2,\ldots\}$ and $\zz_+:=\nn\cup\{0\}$.
We also use
$C_{(\alpha,\beta,\ldots)}$ to denote a positive
constant depending on the indicated parameters $\alpha,
\beta,\ldots.$ The symbol $f\lesssim g$ means
that $f\le Cg$. If $f\lesssim g$ and $g\lesssim f$,
we then write $f\sim g$. If $f\le Cg$ and $g=h$ or
$g\le h$, we then write $f\ls g\sim h$ or $f\ls g\ls h$,
\emph{rather than} $f\ls g=h$ or $f\ls g\le h$.
We use $\mathbf{0}$ to denote the \emph{origin} of $\rn$.
If $E$ is a subset of $\rn$, we denote by $\mathbf{1}_E$
its characteristic function.
For any $\alpha\in(0,\infty)$ and any ball $B:=B(x_B,r_B)$
in $\rn$, with $x_B\in\rn$ and $r_B\in(0,\infty)$, we
let $\alpha B:=B(x_B,\alpha r_B)$. Finally, for any
$q\in[1,\infty]$, we denote by $q'$ its
\emph{conjugate exponent}, namely, $1/q+1/q'=1$.
\section{Proof of Theorem \ref{thm-2-3}}\label{section-3}
In this section, we establish the characterization of
the Sobolev semi-norm on  the  metric measure space of homogeneous type, $(\cx, \rho, \mu).$

\subsection {Proof of Theorem \ref{thm-2-3}: Lower Estimate}\label{subsection2.1}
In this subsection, we prove the lower bound
of Theorem \ref{thm-2-3}. Indeed, the stated
lower bound in Theorem \ref{thm-2-3} follows
directly from the following theorem.
\begin{theorem}\label{thm-3-1}
	Let $(\cX, \rho, \mu)$ be a  metric measure space of homogeneous type and
	$p\in[1,\infty)$.   Then there exists
	a positive constant $C,$ depending only on $p$ and the doubling constant of $\mu$, such that,  for any $f\in \cC_{\mathrm{c}}^\ast(\cX),$
	\begin{equation}\label{7-1}
	\liminf_{\ld\to \infty} \ld^p \int_{\cX} \int_{\cX} \mathbf{1}_{D_\ld}
(x, y) \,d\mu(x) \,d\mu(y)	\ge C  \int_{\cX} [\blip f(x)] ^p  \,d\mu(x),
	\end{equation}
	where $D_\ld$ is as in \eqref{keydef}.
\end{theorem}
Before we prove Theorem \ref{thm-3-1}, we need
the following lemma, which establishes  the geometrical relation of differences and derivatives in
the metric measure space.
\begin{lemma}\label{lem-2-33}
Let $(\cX, \rho, \mu)$ be a  metric measure space of homogeneous type
and  $f\in \cC_{\mathrm{c}}^\ast(\cX)$.  Then, for any
given $\da\in (0, 1),$ there exists a
positive constant $r_\da \in (0, \da),$ depending only on $\da$
and $f,$ such that, for  any $x\in \cX$ with $\blip f(x)\in(4\da,\infty)$,
and any $r\in(0,r_\da]$,
	\begin{equation*}
		\mu(S(x,r))\ge c_0 \mu(B(x,r)),
	\end{equation*}
where the positive constant $c_0$ depends only on the doubling constant of $\mu$ and
	\begin{equation}\label{2-1}
		S(x,r):=\left\{ z\in B(x, r):\   |f(z)-f(x)| \ge \f1{8}  \blip f(x)   \rho(z,x)\right\}.
	\end{equation}
\end{lemma}
\bp
Let all the notation be as in the present lemma. Let $g:=\blip f$ and
$$A_\da:=\left\{x\in \cX:\
g(x)\in(4\da,\infty)\right\}.$$
By the definition of g, we find that there exists a
constant $r_\da\in(0,\da)$ such that, for
any $x,\ y\in\cx$ with $\rho(x,y)\in[0,r_\da]$,
\begin{equation}\label{2-3-0}
	|g(x)-g(y)|<\da
\end{equation}
and, for any $r\in(0,r_\da],$
\begin{equation}\label{2-4-0}
	g(x)-\da < \sup_{y\in B(x,r)}\f {|f(y)-f(x)| } {r} < g(x) +\da.
\end{equation}
Let  $x\in A_\da$ and $r\in(0,r_\da]$ be fixed.
Using the first inequality in \eqref{2-4-0}
and the definition of $A_\delta$,
we conclude that there exists a point $y_{x,r}
\in B(x, r/2)$ such that
\begin{align}\label{6.30}
|f(y_{x,r})-f(x)|> \f12[g(x)-\da]r >\f 38 g(x) r.
\end{align}
Moreover, by \eqref{2-3-0} and the second inequality of \eqref{2-4-0}, we have
\[ \sup_{ z\in B(y_{x,r}, r/6)} |f(z)-f(y_{x,r})|
< \f16\left[g(y_{x,r} ) +\da\right]r
<  \f16\left[g(x ) +2\da\right]r
< \f 1{4} g(x) r.\]
From this and \eqref{6.30}, it follows that,
for any $ z\in B(y_{x,r}, r/6)\subset B(x,r)$,
\begin{align*}
	|f(z)-f(x)|\ge |f(y_{x,r}) -f(x)|-|f(z)-f(y_{x,r})|>\f18 g(x) r>
	 \f 1 {8} g(x) \rho(x,z).
\end{align*}
This implies that
\[B(y_{x,r}, r/6) \subset S(x,r)\subset B(x,r) \subset B(y_{x,r}, 2r), \]
which, together with the doubling property of the measure $\mu$, further implies that
\[ \mu(S(x,r)) \ge \mu (B(y_{x,r}, r/6) )
\gtrsim \mu (B(x,r)).\]
This finishes the proof of Lemma \ref{lem-2-33}.
\ep

Let
$p\in[1,\infty].$ We denote by the \emph{symbol} $L^p(\mu)$ the
Lebesgue $L^p$-space defined with respect to
the measure $\mu$ on ${\mathcal{X}}$.
Recall that, for any $f\in L^p(\mu)$,
$$
\|f\|_{L^p(\mu)}:=
\left[\int_{\mathcal{X}}|f(x)|^p\,d\mu(x)\right]^{1/p}
$$
with the usual modification made when $p=\infty$.

\bp[Proof of Theorem \ref{thm-3-1}]
Let all the notation be as in the present theorem. By the Tonelli theorem, we have,
for any
$\ld\in(0,\infty),$
\begin{align}
\int_{\cX} \int_{\cX} \mathbf{1}_{D_\ld} (x, y)\,d\mu(x)\,d\mu(y) &=\int_{\cX}  \mu\left( D_\ld(x)\right) \, d\mu(x),\label{2-4-b}
\end{align}
where
$$
D_\ld(x):=\left\{ y\in\cX:\ (x,y)\in D_\ld\right\}.$$
For  a sufficiently small  constant $\da\in
(0, 1)$, let
$$G_{\da}:=\left\{ x\in\cX:\ \blip f(x)\in(4\da,4\da^{-1})\right\}.$$
Let $r_\da\in (0, \da)$ be as
in Lemma \ref{lem-2-33}. For any $x\in G_\da$
and $\ld\in(0,\infty),$ let $$\mathrm{I}_\da(x,\ld):=\left\{t\in(0,r_\da]:\
[\blip f(x)]^p >  8^p\ld^p \mu(B(x,t))\right\}$$
and
\begin{align*} D_{\ld,1}(x):&=\bigcup_{t\in\mathrm{I}_\da(x,\ld)}S(x,t),
\end{align*}
where  $S(x,t)$ is as in \eqref{2-1}.
From this, it follows that, for any $x\in G_\da$
and $\ld\in(0,\infty),$  \begin{equation}\label{2-5-0}
	D_\ld(x)\supset  D_{\ld,1}(x).
\end{equation}
Recall that, for any $x\in\cx$ and
$t\in(0,\infty),$ $\mu(B(x,t))\in(0,\infty)$
and $\mu(\{x\})=0.$ Thus,  the function $
\mu(B(x,\cdot))$ is non-decreasing   on $(0,\infty)$ for any  given
$x\in\cX$, and satisfies
$$\lim_{t\in(0,\infty),t\to0}
\mu(B(x,t))=0\   \ \text{ and}\   \
\lim_{t\to\infty} \mu(B(x,t))=\mu(\cX).$$
This implies that,  for any $x\in
G_\da$ and $\ld\in[\da^{-1}[\mu(\cx)]^{-1/p},\infty)$,
there exists a unique  integer  $N(x,\ld)$
such that
\begin{equation}\label{2-3-a}
	\mu(B(x,2^{-N(x,\ld)}))< 8^{-p}\ld^{-p} [\blip f(x)] ^p \leq \mu(B(x,2^{-N(x,\ld)+1})).
\end{equation}
Clearly,  the function $N(x, \cdot)$ is
non-decreasing  on $[\da^{-1}[\mu(\cx)]^{-1/p},\infty)$ for any
given $x\in G_\da$. For any $x\in G_\da$ and $\ld\in[\da^{-1}[\mu(\cx)]^{-1/p},\infty),$
let $\va_\ld(x):=2^{-N(x,\ld)}$.
We  claim that
\begin{equation}\label{2-3}
	\lim_{\ld\to\infty} \sup_{x\in G_\da} \va_\ld(x) =0.
\end{equation}
For the moment, we take this claim for
granted, and proceed with the proof of the present
theorem. Using \eqref{2-3-a}, we conclude
that, for any $x\in G_\da,$ $\ld\in[\da^{-1}[\mu(\cx)]^{-1/p},\infty),$ and
$t\in(0,\va_\ld(x)],$
\[ [\blip f(x)] ^p > 8^p\ld^p \mu(B(x,\va_\ld(x)))\ge 8^p\ld^p \mu(B(x,t)).\]
This implies that, for any  $x\in G_\da$ and
$\ld\in[\da^{-1}[\mu(\cx)]^{-1/p},\infty),$
\begin{align}\label{2302}
D_{\ld,1} (x) \supset
\bigcup_{t\in(0,\min\{r_\da, \va_\ld(x)\}]}
S(x,t)=S(x,\min\{r_\da, \va_\ld(x)\}).
\end{align}
On the other hand, by the claim \eqref{2-3}, we conclude that there exists a constant  $$\ld_\da\in\lf[\da^{-1}[\mu(\cx)]^{-1/p},\infty\r)$$
such that, for any  $\ld\in(\ld_\da,\infty)$,
\begin{align}\label{1639}
\sup_{x\in G_\da} \va_\ld(x)<r_\da.
\end{align}
This, together with \eqref{2302}, implies that,  for any $x\in G_\da$ and $\ld\in(\ld_\da,\infty)$,
\begin{align}\label{2-5} D_{\ld,1}(x)&\supset S(x, \va_\ld(x)).
\end{align}
By \eqref{2-5-0}, \eqref{2-4-b}, and \eqref{2-5}, we find that, for any $\ld\in(\ld_\delta,\infty)$,
\begin{align}
	\int_{\cX} \int_{\cX} \mathbf{1}_{D_\ld} (x, y) \,d\mu(x) \,d\mu(y)
	\ge \int_{G_\da} \mu \left(   S\bigl(x,  \va_\ld(x)\bigr)\right)\, d\mu(x).\label{2-6-eq}
\end{align}
Using Lemma \ref{lem-2-33}, \eqref{1639}, \eqref{1133}, and
\eqref{2-3-a}, we conclude that, for any
$x\in G_\da$ and  $\ld\in(\ld_\da,\infty)$,
\begin{align*}
	\mu\left(S(x, \va_\ld(x))\right )\gtrsim \mu(B(x,\va_\ld(x)))\gtrsim
\mu(B(x,2\va_\ld(x)))\gtrsim\ld^{-p}  [\blip f(x)] ^p.
\end{align*}
From this and \eqref{2-6-eq},  it follows
that, for any $\ld\in(\ld_\da,\infty),$
\begin{align*}
\ld^p  \int_{\cX} \int_{\cX}
\mathbf{1}_{D_\ld} (x, y) \,d\mu(x) \,d\mu(y)
\gtrsim\int_{G_\da} [\blip f(x)] ^p \,d\mu(x).
\end{align*}
This implies that
\begin{align*}
	\liminf_{\ld\to\infty} \ld^p  \int_{\cX} \int_{\cX} \mathbf{1}_{D_\ld} (x, y) \,d\mu(x) \,d\mu(y)
	&\gtrsim \lim_{\da\in(0,1),\da\to0}  \int_{G_\da}[\blip f(x)] ^p  \,d\mu(x)\\
	&\sim\int_{\cX} [\blip f(x)] ^p \,d\mu(x),
\end{align*}
which completes the proof of \eqref{7-1}.

It remains to show the claim \eqref{2-3}. Notice
that the function
$\sup_{x\in G} 2^{-N(x,\cdot)}$
is non-increasing on $(0,\infty)$ and hence the
limit $$ \lim_{\ld\to \infty} \sup_{x\in G_\da}
2^{-N(x,\ld)}$$ exists and is nonnegative.
Assume to the contrary that \eqref{2-3} is not true.
Then there exists a constant $\da_0\in (0,1)$ such
that, for any $\ld\in(0,\infty),$
\begin{align}\label{9001}
	\sup_{x\in G_\da} 2^{-N(x,\ld)}>\da_0>0.
\end{align}
For any $\ld\in(0,\infty)$, let $x_\ld\in G_\da$ be such
that $ 2^{-N(x_\ld,\ld)}>\da_0$. Next, let
$x_0\in\cx$ and $M\in(1, \infty)$ be such that $\supp
f\subset B(x_0, M)$.
Thus, $\rho(x_0,x_\ld)\in[0,M).$
Suppose $L_{\mu}\in(1,\infty)$ and let $s_\mu:=\log_2 L_\mu$.
Since $\mu$ is the doubling measure, we deduce that, for
any $r\in(0,\infty)$ and $x,\ y\in \cx$,
\[ \mu(B(x,r)) \lesssim\lf[1+r^{-1} \rho(x,y)\r]^{s_\mu}\mu(B(y,r)).\]
This, combined with \eqref{1133},
the fact that
$\rho(x_0,x_\ld)\in[0,M),$
\eqref{9001}, and \eqref{2-3-a}, implies
that, for any $\ld\in(0,\infty)$,
\begin{align*}
	0&<\mu(B(x_0, M))
	\lesssim\lf[1+M^{-1}\rho(x_0,x_\ld)\r]^{s_\mu}\mu(B(x_\ld,M))\\
	&\lesssim
	\mu(B(x_\ld,\delta_0)) \lesssim\mu(B(x_\ld,2^{-N(x_\ld,\ld)}))
	\lesssim
	\ld^{-p}\|\blip f\|_{L^\infty(\mu)}^p.
\end{align*}
Let $\ld\to\infty.$ Then we have $\mu(B(x_0,M))=0.$ This
contradicts to
$\mu(B(x_0,M))\in(0,\infty)$.
Thus, the claim \eqref{2-3} holds
true. This finishes the proof of Theorem \ref{thm-3-1}.
\ep
\subsection {Proof of Theorem \ref{thm-2-3}: Upper Estimate}\label{subsection2.2}
In this subsection, we prove the stated upper
bound of Theorem \ref{thm-2-3}.
The following lemma is just \cite[Theorem 1.2]{H01}.
\begin{lemma}\label{Lemma2.3}
	Let $\mathcal{F}$ be any collection of balls with
	uniformly bounded diameter in a metric space
	$(\cx,\rho)$. Then there exists a  family
	$\mathcal{G}$ of disjoint balls in $\mathcal{F}$ such that
	$$\bigcup_{B\in\mathcal{F}}B\subset\bigcup_{B\in\mathcal{G}}5B.$$
\end{lemma}
The upper estimate of Theorem \ref{thm-2-3}
follows directly from the following theorem.
\begin{theorem}\label{thm-2-2}
Let $p,$ $q\in[1,\infty)$ and $(\cx, \rho,
\mu)$ be a metric space of homogeneous type.
Assume that $f\in\Lip(\cx)$ satisfies a $(q,p)$-Poincar\'e inequality
with constants $C_1,$ $C_2,$
$\tau\in[1,\infty)$.	
Then there exists a positive constant $C$
such that	
\begin{equation}\label{2-9-eq}
	\sup_{\ld\in(0,\infty)} 	\ld^p \int_
	{{\mathcal{X}}}\int_{{\mathcal{X}}} \mathbf{1}
	_{D_\ld}(x, y) \,d\mu(x) \,d\mu(y) \leq C	
	\int_{{\mathcal{X}}}[\lip f(x)]^p\, d\mu(x),
\end{equation}
where $D_\ld$ is as in \eqref{keydef} for any $\ld\in(0,\infty)$, and
the constant $C$ depends only on $p,\,q$, the doubling
constant of $\mu$, and the constants $C_1,$ $C_2,$ and $\tau$.
\end{theorem}	

Now, we prove Theorem \ref{thm-2-2}. To this end,
  we need the following lemma.
\begin{lemma}\label{lem-2-3}
Let $\mathcal{X}$, $f$, $C_1$, $C_2$, $p$,  and $q$ be as in Theorem
\ref{thm-2-2}. Let $\tau$ and $\{ \ell_B \}_{B\in\BB}$
be as in Definition \ref{45622}. Assume that $B=B(x,r)$ is a ball with $x\in {\mathcal{X}}$
and $r\in(0,\infty),$
and $B_1\in \BB$ is
a ball such that $x\in B_1\subset B\subset \al B_1$ with $\al\in[1,\infty)$. Then
there exists a positive constant $C$ such that
\begin{equation}\label{5-1b}
	|f(x)-\ell_{B_1}(f)|^p\leq Cr^p\sum_{j=0}^\infty2^{-j}
	\av_{2^{-j}\tau B}[\lip f(z)]^p\,d\mu(z),
\end{equation}
where the positive constant $C$ depends only on $p,\,q,\,\al$,
the doubling constant of $\mu$, and the constants $C_1$ and $C_2$.	
\end{lemma}
\begin{proof}
Let all the notation be as in the present lemma. Let $x\in {\mathcal{X}},$
$r\in(0,\infty),$ and $B:=B(x,r).$
By \eqref{3-9-a} and $f\in BC(\cx)$, we have
\begin{align*}
|\ell_{2^{-j}B}(f)-f(x)|&=|\ell_{2^{-j}B}(f-f(x))|
\lesssim\left[\frac{1}{\mu(2^{-j}B)}\int_{2^{-j}B}|f(y)-f(x)|
^q\,dy\right]^{\frac{1}{q}}\\
&\lesssim\sup_{y\in 2^{-j}B}|f(y)-f(x)|.
\end{align*}
This implies that
$$f(x)=\lim_{j\to\infty}\ell_{2^{-j }B}f.$$
From this, \eqref{3-9-a}, \eqref{1133}, and \eqref{2-8-eq},
it follows that
\begin{align*}
|f(x)-\ell _{B}(f)|
&=\lim_{j\to\infty}|\ell_{2^{-j}B}f(x)-\ell_{B}f(x)|\lesssim\sum_{j=0}^\infty|\ell_{2^{-j-1}
B}(f)-\ell_{2^{-j}B}(f)|\\
&\sim\sum_{j=0}^\infty\left|\ell_{2^{-j-1} B}\lf(f-\ell
_{2^{-j}B}(f)\r)\right|\\
&\lesssim\sum_{j=0}^\infty\left[\av_{2^{-j-1}B}
|f(y)-\ell_{2^{-j}B}(f)|^q\, d\mu(y)\right]^{\f1q}\\
&\lesssim \sum_{j=0}^\infty
\left[\av_{2^{-j}B}|f(y)-\ell_{2^{-j} B}(f)|
^q\, d\mu(y)\right]^{\f1q}\\
&\lesssim\sum_{j=0}^\infty2^{-j}
r\left\{\av_{2^{-j}\tau B} [\lip f(y)]^p\, d\mu(y)\right\}^{\f1p}.
\end{align*}
Using this and the H\"older inequality, we
find that
\begin{align}\label{3-13}
|f(x)-\ell_B(f)|^p &\lesssim r^p\sum_{j=0}^\infty2^{-j}
\av_{2^{-j}\tau B}[\lip f(z)]^p\,d\mu(z).
\end{align}
By \eqref{3-9-a}, \eqref{1133}, and
\eqref{2-8-eq}, we conclude that
\begin{align*}
|\ell_B(f)-\ell_{B_1}(f)|^p
&\sim|\ell_{B_1}(f-\ell_B(f))|^p\lesssim\left[\av_{B_1}
|f(z)-\ell_B(f)|^q\, d\mu(z)\right]^{p/q}\\
&\lesssim\left[\av_{B}|f(z)-\ell_B(f)|
^q\, d\mu(z)\right]^{p/q}\lesssim r^p
\av_{\tau B} [\lip f(z)]^p\, d\mu(z).
\end{align*}
This, combined with \eqref{3-13}, implies \eqref{5-1b},
which completes the proof of Lemma \ref{lem-2-3}.
\end{proof}
\bp [{Proof of Theorem \ref{thm-2-2}}]
Let all the notation be as in the present theorem. For any $(x,y)\in{\mathcal{X}}\times{\mathcal{X}}$, let
$z_{x,y}$ be an arbitrarily given point of
$B(x, \rho(x,y))$ and
$$ B_{x,y}:=B(z_{x,y}, 2\rho(x,y)).$$
Clearly, for any $x,$ $y\in\cx,$
\[[\rho(x,y)]^{p} V(x,y)\sim\mu(B_{x,y})r_{B_{x,y}}^p.\]
From this and the fact that, for any $x,$ $y\in\mathcal{X},$
\begin{align*}
|f(x)-f(y)|\leq|f(x)-\ell_{B_{x,y}}(f)|+|
f(y)-\ell_{B_{x,y}}(f)|,
\end{align*}
it follows that, for any $\lambda\in(0,\infty),$
$$ D_\ld \subset D_\ld^1\cup D_\ld^2,$$
where
$$D_\ld^1:=\lf\{(x,y)\in {\mathcal{X}}\times
{\mathcal{X}}:\ |f(x)-\ell_{B_{x,y}} (f)|
^p>c\ld^p \mu(B_{x,y})r_{B_{x,y}}^p\r\}$$
and
$$D_\ld^2:=\lf\{ (x,y)\in {\mathcal{X}}\times
{\mathcal{X}}:\ |f(y)-\ell_{B_{x,y}} (f)|
^p>c\ld^p \mu(B_{x,y}) r_{B_{x,y}}^p \r\}$$
with the positive constant $c$ depending only on
$\mu.$ Let $x_0\in {\mathcal{X}}$ be a fixed point.
For any $N\in\NN$, let
$${\mathcal{X}}_N:=B(x_0, N)$$
and, for any $i\in\{1,2\}$,
$$D_{\ld, N}^i:=D_{\ld}^i\cap({\mathcal{X}}_N
\times{\mathcal{X}}_N).$$
By symmetry and the monotone convergence theorem,
we find that, to prove \eqref{2-9-eq}, it suffices
to show that, for any $N\in\NN$ and $\ld\in(0,\infty),$
\begin{equation}\label{3-14-a}
\int_{{\mathcal{X}}\times {\mathcal{X}}}
\mathbf{1}{_{D_{\ld,N}^1}}(x, y) \,d\mu(x)\, d\mu(y)
\lesssim\f 1 {\ld^p}\int_{{\mathcal{X}}}[\lip f(z)]^p\, d\mu(z),
\end{equation}
where the implicit positive constant depends only on
$p,$ $q$, the doubling constant of $\mu$, and the
constants $C_1,$ $C_2,$ and $\tau$ as in the present theorem.

For any $(x,y)\in D_{\ld,N}^1$, let $$\wt B_{x,y}
:=B(x, 3\rho(x,y)).$$ Clearly,
$$B_{x,y}\subset\wt B_{x,y}\subset 3 B_{x,y}
\ \text{and}\ r_{\wt B_{x,y}}:=3\rho(x,y)\leq 6N.$$
By this, \eqref{5-1b}, and \eqref{1133}, we
have, for any $\lambda\in(0,\infty)$ and
$(x, y)\in D_{\ld,N}^1,$
\begin{align*}
&\sum_{j=0}^\infty 2^{-j}\av_{2^{-j}\tau \wt B_{x,y}}
[\lip f(z)]^p\,d\mu(z)\ge C^{-1} r_{\wt B_{x,y}}^{-p}|f(x)-
\ell_{B_{x,y}}(f)|^p>c_1\ld^p \mu(\wt B_{x,y}),
\end{align*}
where $C$ is as in \eqref{5-1b} and $c_1
:=C^{-1}c(\frac{1}{2})^pL_{\mu}^{-2}.$
This implies that, for any $\lambda\in(0,\infty)$
and $(x,y)\in D_{\ld,N}^1,$ there
exists a $j\in\zz_+ $ such that
\begin{equation}\label{maoxiandao}
\av_{2^{-j}\tau\wt B_{x,y}}[\lip f(z)]^p\,d\mu(z)
>c_22^{j/2}\ld^p\mu(\wt B_{x,y}),
\end{equation}
where $c_2:=(1-2^{-\frac{1}{2}})c_1.$
For any $\ld\in(0,\infty)$ and $j\in\zz_+$,
we denote by the \emph{symbol} $\mathcal
{B}_j^N(\ld)$ the collection of all balls $B
\subset {\mathcal{X}}$ with radius $\leq 6N$ such that
\begin{equation}\label{5-7b}
\av_{2^{-j}\tau B}[\lip f(z)]^p\,d\mu(z)>c_22^{j/2}\ld^p\mu(B).
\end{equation}
For any $\ld\in(0,\infty)$ and $j\in\zz_+$, let
\begin{align}\label{dwtl}
	D_j^N(\ld)&:=\lf\{ (x, y)\in{\mathcal{X}}
	_N\times {\mathcal{X}}_N:\ \exists\,
	B\in\mathcal{B}_j^N(\ld)\ \text{such that}
	\ x\in 2^{-j}\tau B\ \text{and}\ y\in \tau B\r\}.
\end{align}
For any $\ld\in(0,\infty)$ and $(x,y)\in D_{\ld,N}^1$,
by \eqref{maoxiandao}, we conclude that there exists
a $j\in\zz_+$ such that $\wt B_{x,y}\in
\mathcal{B}^N_j(\ld)$ and $(x, y)\in (2^{-j}\tau
\wt B_{x,y})\times (\tau\wt B_{x,y})$.
This implies that, for any $\lambda\in(0,\infty),$
\begin{align*}
D_{\ld,N}^1 \subset \bigcup_{j=0}^\infty D_j^N(\ld).
\end{align*}
Thus, to prove \eqref{3-14-a}, it suffices to show that,
for any $N\in\NN$ and $\ld\in(0,\infty),$
\begin{equation}\label{5-8b}
\sum_{j=0}^\infty \int_{{\mathcal{X}}_N}\int
_{{\mathcal{X}}_N} \mathbf{1}_{D_j^N(\ld)} (x,y)
\,d\mu(x)\, d\mu(y) \lesssim \f 1 {\ld^p} \int
_{{\mathcal{X}}} [\lip f(z)]^p\, d\mu(z),
\end{equation}
where the implicit positive constant depends only on
$p,$ $q$, the doubling constant of $\mu$, and the
constants $C_1,$ $ C_2,$ and $\tau$ as in the present theorem.

By \eqref{dwtl}, we find that, for any $j\in\zz_+$
and $\ld\in(0,\infty),$
\begin{equation}\label{3-17b}
D_j^N(\ld)\subset\bigcup_{B\in \mathcal{B}_j^N(\ld)}
\lf( (2^{- j} \tau B )\times (\tau B)\r).
\end{equation}
Now, consider the following metric $\rho_j$ on ${\mathcal{X}}
\times {\mathcal{X}}$, defined by setting, for any
$(x,y),\ (u,v)\in {\mathcal{X}}\times {\mathcal{X}},$
$$ \rho_j\lf( (x,y), (u,v)\r):=\max \lf\{ 2^j\tau^{-1}
\rho(x,u), \ \tau^{-1}\rho(y,v)\r\}.$$
Notice that, for any ball $B:=B(x_0, r)\subset {\mathcal{X}}$ with $r\in(0,\infty)$,
the set $(2^{-j} \tau B)\times(\tau B)$ is a ball with
center $(x_0, x_0)\in {\mathcal{X}}\times {\mathcal{X}}$
and radius $ r$ in the metric space $({\mathcal{X}}
\times {\mathcal{X}},\rho_j)$. Thus, $\{(2^{-j} \tau B)
\times (\tau B)\}_{B\in \mathcal{B}_j^N}$ is a family
of balls in the metric space $({\mathcal{X}}\times
{\mathcal{X}},\rho_j)$ with radius not more than
$ 6 N$.   Applying Lemma \ref{Lemma2.3}
to the collection of the balls,
$$\lf\{ (2^{-j}\tau B)\times (\tau B):\
B\in\mathcal{B}^N_j(\ld)\r\},$$
in the metric space $({\mathcal{X}}\times
{\mathcal{X}},\rho_j)$, we conclude that there
exists a family $\mathcal{G}_j$ of disjoint balls in $\mathcal{B}_j^N(\ld)$ such that
\begin{equation}\label{5-10b}
\bigcup_{B\in\mathcal{B}_j^N(\ld)} \lf((2^{-j}
\tau B)\times \tau B \r) \subset \bigcup_{B\in\mathcal{G}_j} (5\cdot 2^{-j}
\tau B)\times (5\tau B),
\end{equation}
and the balls $\{(2^{-j}\tau B)\times \tau B\}_{B\in\mathcal{G}_j}$ are pairwise disjoint, which
in turn implies that the balls
$$\lf\{2^{-j}\tau B\r\}_{B\in\mathcal{G}_j}$$
are pairwise disjoint.
By \cite[p.\,67]{CW71} and \cite[Lemma 2.5]{H2010}, we find that any disjoint
collection of balls in $\cx$ is at most countable. Thus, the sets $\{2^{-j}\tau B\}_{B\in\mathcal{G}_j}$
is at most countable and we rewrite $\{2^{-j}\tau B\}_{B\in\mathcal{G}_j}$
by  $\{2^{-j}\tau B_{j,k}\}_{k\in\nn}$. From this and \eqref{5-10b},  it follows that,
for any $\ld\in(0,\infty),$
\begin{align*}
&\sum_{j=0}^\infty \int_{{\mathcal{X}}_N\times
{\mathcal{X}}_N} \mathbf{1}_{D_j^N(\ld)}(x,y)\,d\mu(x)\, d\mu(y)\\
&\quad\lesssim\sum_{j=0}^\infty \sum_{k} \mu(5\tau B_{j,k})
\mu(5\cdot2^{-j}\tau B_{j,k})\lesssim\sum_{j=0}^\infty \sum_{k} \mu(B_{j,k})
\mu(2^{-j}\tau B_{j,k}) \\
&\quad\lesssim \ld^{-p}\sum_{j=0}^\infty 2^{-j/2}
\sum_{k} \int_{2^{-j}\tau B_{j,k}} [\lip f(z)]^p \, d\mu(z)\\
&\quad\lesssim \ld^{-p} \sum_{j=0}^\infty 2^{-j/2}
\int_{{\mathcal{X}}} [\lip f(z)]^p \, d\mu(z)\lesssim
\ld^{-p } \int_{{\mathcal{X}}} [\lip f(z)]^p\, d\mu(z),
\end{align*}
where we used \eqref{3-17b}, \eqref{5-10b}, and
\eqref{1133} in the first
step, \eqref{5-7b} and the fact that $B_{j,k}
\in\mathcal{B}_j^N(\ld)$ in the second step,
the fact that the balls $\{2^{-j}\tau B_{j,k}\}_{k\in\mathbb{N}}$
 are pairwise disjoint in
the third step, and the implicit positive constant
depends only on $p,$ $q$, the doubling constant of
$\mu$, and the constants $C_1,$ $C_2,$ and $\tau$
as in the present theorem.
This proves \eqref{5-8b}, which completes the
proof of Theorem \ref{thm-2-2}.
\ep
\section{ Fractional Sobolev
	and Gagliardo--Nirenberg Type Inequalities}
\label{section2}
As an application of Theorem \ref{thm-2-3},
we obtain the following fractional Sobolev
and Gagliardo--Nirenberg type inequalities in
any metric
measure space of homogeneous type.
\begin{corollary}\label{corollart6.5}
Let $(\cX, \rho, \mu)$ be a  metric
measure space of homogeneous type and $p,\,q\in[1,\infty).$ Assume that $f\in \cC_{\mathrm{c}}^\ast(\cx)$
satisfies a $(q,1)$-Poincar\'e inequality
with constants $C_1,$ $ C_2,$ $ \tau\in[1,\infty)$
Then there
exists a
positive constant $C$, depending on $C_1,\,C_2,\,p,\,q,$ and $\tau,$ such that
\begin{align}\label{hehe1}
&\sup_{\ld\in(0,\infty)}\ld\left[
\int_{\cx}\mu\lf(\lf\{x\in\cx:\ |f(x)-f(y)|>\ld[\rho(x,y)]^{\f1p}[V(x,y)]^{\f1p}\r\}\r)\,d\mu(y)\right]^{\f1p}\\
&\quad\leq C\|f\|_{L^\infty(\mu)}^{1-\f1p}\|\blip f\|_{L^1(\mu)}^{\f1p},\noz
\end{align}
where $V(x,y):=\mu(B(x,\rho(x,y)))$ for any $x,\ y\in\cx$.
\end{corollary}
\begin{proof}
Let all the notation be as in the present corollary. For any $\ld, \,r, \,s \in(0,\infty),$  let
\begin{align}\label{2104}
E_f(\ld,r,s):=\left\{(x,y)\in\cx\times\cx:\
|f(x)-f(y)|>\ld[\rho(x,y)]^s[V(x,y)]^{\f1r}\right\}.
\end{align}
Since, for any $\ld\in(0,\infty),$
$$E_f(\ld,p,1/p)\subset E_f\left(\frac{\ld^p}{[2\|f\|_{L^\infty(\mu)}]^{p-1}},1,1\right),$$
from Theorem \ref{thm-2-3}, we deduce that
\begin{align*}
&\sup_{\ld\in(0,\infty)}\ld\left[
\int_{\cx}\mu\lf(\lf\{x\in\cx:\ |f(x)-f(y)|>\ld[\rho(x,y)]^{\f1p}[V(x,y)]^{\f1p}\r\}\r)\,d\mu(y)\right]^{\f1p}\\
&\quad\lesssim\sup_{\ld\in(0,\infty)}\ld\left[
\int_{\cx}\mu\lf(\lf\{x\in\cx:\ |f(x)-f(y)|>\ld^p/[2\|f\|_{L^\infty(\mu)}]^{p-1}\rho(x,y)V(x,y)\r\}\r)\,d\mu(y)\right]^{\f1p}\\
&\quad\sim\|f\|_{L^\infty(\mu)}^{1-\f1p}\sup_{\ld\in(0,\infty)}\left[\ld
\int_{\cx}\mu(\{x\in\cx:\ |f(x)-f(y)|>\ld\rho(x,y)V(x,y)\})\,d\mu(y)\right]^{\f1p}\\
&\quad\lesssim
\|f\|_{L^\infty(\mu)}^{1-\f1p}\|\blip
f\|_{L^1(\mu)}^{\f1p}.
\end{align*}
This finishes the proof of Corollary \ref{corollart6.5}.
\end{proof}
\begin{remark}
In the case when $\cx:=\rn,$ $d\mu(x):=dx$, and $\rho(x,y):=|x-y|$ for
any $x,\ y\in\rn,$  Corollary \ref{corollart6.5}
is just \cite[Corollary 5.1]{ref8}.
\end{remark}
\begin{corollary}\label{corollary6.6}
Let $s_1 \in(0,1)$, $p_1\in(1,\infty),$ and
$\ta\in (0,1)$. Let $s\in (s_1, 1)$ and
$p\in (1, p_1)$ satisfy $s:=(1-\theta) s_1+\theta$
and $ \frac{1}{p}:=\frac{1-\theta}{p_1}+\theta$.
Let $(\cX, \rho, \mu)$ be a  metric
measure space of homogeneous type. Assume that $f\in \cC_{\mathrm{c}}^\ast(\cx)$ satisfies a $(q,1)$-Poincar\'e inequality
with constants $C_1,$ $ C_2,$ $ \tau\in[1,\infty)$.
Then there exists a positive constant $C,$
depending on $C_1,\,C_2,\,p,\,q,$ and $\tau,$  such that
\begin{align}\label{hehe2}
&\sup_{\ld\in(0,\infty)}\ld\left[
\int_\cx\mu\lf(\lf\{x\in\cx:\ |f(x)-f(y)|>\ld[\rho(x,y)]^{s}[V(x,y)]^{\f1p}\r\}\r)\,d\mu(y)\right]^{\f1p}\\\noz
&\quad\leq C\|\blip
f\|_{L^1(\mu)}^{\theta}\left[\int_{\cx}
\int_{\cx}\frac{|f(x)-f(y)|^{p_1}}{[\rho(x,y)]^{s
_1p_1}V(x,y)}\,d\mu(x)\,d
\mu(y)\right]^{\frac{1-\theta}{p_1}},
\end{align}
where $V(x,y)=\mu(B(x,\rho(x,y)))$ for any $x,y\in\cx$.
\end{corollary}
\begin{proof}
Let all the notation be as in the present corollary. Let $A\in(0,\infty)$ be a constant
which is specified later, and $E_f(\ld,r,s)$
as in \eqref{2104}. Since, for any $x,$ $ y\in\cx,$
$$\frac{|f(x)-f(y)|}{[\rho(x,y)]^s[V(x,y)]^{1/p}}
=\left\{\frac{|f(x)-f(y)|}{[\rho(x,y)]^{s_1}[V(x,y)]^{1/p_1}}\right\}^{1-\theta}
\left[\frac{|f(x)-f(y)|}{\rho(x,y)V(x,y)}\right]^{\theta},$$
it follows that, for any $\ld\in(0,\infty),$
\begin{align*}
	&E_f(\ld, p,s) \subset \left[E_f \left(A^{-\ta}\ld,
	p_1, s_1\right) \cup E_f\left(A^{1-\ta}\ld, 1,1\right)\right].
\end{align*}
This implies that, for any $\ld\in(0,\infty),$
\begin{align}\label{2207}
&\left[\int_\cx\mu(\{x\in\cx:\ (x,y)\in E_f(\ld,p,s)\})\,d\mu(y)\right]^{\f1p}\\\noz
&\quad\lesssim \left[\int_\cx\mu\lf(\lf\{x\in\cx:\ (x,y)\in E_f(A^{-\theta}\ld,p_1,s_1)\r\}\r)\,d\mu(y)\right]^{\f1p}\\\noz
&\quad\quad+\left[\int_\cx\mu\lf(\lf\{x\in\cx:\ (x,y)\in E_f(A^{1-\theta}\ld,1,1)\r\}\r)\,d\mu(y)\right]^{\f1p}\\\noz
&\quad\lesssim \left[\f{\mathrm{G}}{A^{-\theta}\ld}\right]^{\frac{p_1}{p}}+\left[\f{\mathrm{H}}{A^{1-\theta}\ld}\right]^{\f1p},
\end{align}
where
$$\mathrm{G}:=\sup_{\ld\in(0,\infty)}\ld
\left[\int_\cx\mu\lf(\lf\{x\in\cx:\ |f(x)-f(y)|>\ld[\rho(x,y)]^{s_1}
[V(x,y)]^{\f1{p_1}}\r\}\r)\,d\mu(y)\right]^{\f1{p_1}}$$
and
$$\mathrm{H}:=\sup_{\ld\in(0,\infty)}\ld
\int_\cx\mu(\{x\in\cx:\ |f(x)-f(y)|>\ld\rho(x,y)V(x,y)\})\,d\mu(y).$$
Choose $A\in(0,\infty)$ such that
$$\left(\frac{\mathrm{G}}{A^{-\theta}\ld}\right)^{\frac{p_1}{p}}=\left(\frac{\mathrm{H}}{A^{1-\theta}\ld}\right)^{\f1p}.$$
This,  combined with \eqref{2207}, implies that
\begin{align*}
&\left[\int_\cx\mu(\{x\in\cx:\ (x,y)\in E_f(\ld,p,s)\})\,d\mu(y)\right]^{\f1p}\\\noz	 &\quad\lesssim\left(\frac{\mathrm{G}}{A^{-\theta}
\ld}\right)^{\frac{p_1}{p}}+
\left(\frac{\mathrm{H}}{A^{1-\theta}\ld}
\right)^{\f1p}
\sim\left(\frac{\mathrm{G}}{A^{-\theta}
\ld}\right)^{\frac{p_1}{p}}
\sim\ld^{-1}\mathrm{H}^{\theta}\mathrm{G}^{1-\theta}.
\end{align*}
From this and Theorem \ref{thm-2-3}, we deduce
that
\begin{align*}
&\sup_{\ld\in(0,\infty)}\ld\left[
\int_\cx\mu\lf(\lf\{x\in\cx:\ |f(x)-f(y)|>\ld[\rho(x,y)]^{s}[V(x,y)]^{\f1p}\r\}\r)\,d\mu(y)\right]^{\f1p}\\
&\quad\lesssim\mathrm{H}^{\theta}\mathrm{G}^{1-\theta}\\
&\quad\lesssim\|\blip
f\|_{L^1(\mu)}^{\theta}\left[\int_{\cx}
\int_{\cx}\frac{|f(x)-f(y)|^{p_1}}{[\rho(x,y)]^{s
_1p_1}V(x,y)}\,d\mu(x)\,d
\mu(y)\right]^{\frac{1-\theta}{p_1}}.
\end{align*}
This finishes the proof of Corollary \ref{corollary6.6}.
\end{proof}
\begin{remark}
	In the case when $\cx:=\rn,$ $d\mu(x):=dx$, and $\rho(x,y):=|x-y|$ for
	any $x,\ y\in\rn,$ Corollary \ref{corollary6.6}
	is just  \cite[Corollary 5.2]{ref8}.
\end{remark}

\section{Applications}\label{section3}

In this section, we apply Theorem \ref{thm-2-3} and Corollaries \ref{corollart6.5} and \ref{corollary6.6}
to two concrete examples of metric measure spaces,
namely, $\rn$ with weighted Lebesgue measure
(see Subsection \ref{1r1} below)
and the complete Riemannian $n$-manifold with non-negative Ricci curvature
(see Subsection \ref{1r3} below).

\subsection{$A_p(\rn)$-weight}\label{1r1}
Let us first recall the notion of
Muckenhoupt weights $A_p(\rn)$ (see, for
instance, \cite{ref1}).
\begin{definition}\label{weight}
	An \emph{$A_p(\rn)$-weight} $\omega$, with
	$p\in[1,\infty)$, is a nonnegative locally
	integrable function on $\rn$ satisfying
	that, when $p\in(1,\infty),$
	$$[\omega]_{A_p(\rn)}:=\sup_{Q \subset\rn}
	\left[\frac{1}{|Q|}\int_{Q}\omega(x)\,dx
	\right]\left\{\frac{1}{|Q|}\int_Q[\omega(x)]^
	{\frac{1}{1-p}}\,dx\right\}^{p-1}<\infty,$$
	and
	$$[\omega]_{A_1(\rn)}:=\sup_{Q\subset\rn}
	\frac{1}{|Q|}\int_Q\omega(x)\,dx\left
	[\|\omega^{-1}\|_{L^\infty(Q)}\right]<\infty,$$
	where the suprema are taken over all
	cubes $Q\subset\rn$.
	
	Moreover, let
	$$A_{\infty}(\rn):=\bigcup_{p\in[1,\infty)}A_p(\rn).$$
\end{definition}
\begin{definition}\label{twl}
	Let $p\in[0,\infty)$ and $\omega\in A_{\infty}
	(\rn).$ The \emph{weighted Lebesgue space}
	$L^p_{\omega}(\rn)$ is defined to be the
	set of all measurable functions $f$ on $\rn$
	such that
	$$\|f\|_{L^p_{\omega}(\mathbb{R}^n)}:=\left
	[\int_{\mathbb{R}^n}|f(x)|^p\omega(x)\,dx
	\right]^{\frac{1}{p}}<\infty.$$
\end{definition}
Let $p\in[1,\infty]$ and $\omega\in A_p(\rn)$. For any measurable set $E\subset\rn,$ let $$\omega(E):=\int_{E}\omega(x)\,dx.$$ The following lemma is a part of
\cite[Proposition 7.1.5]{ref1}.
\begin{lemma}\label{Lemma2.1}
	Let $p\in[1,\infty)$ and $\omega\in A_p(\rn).$ Then the following statements hold true.
	\begin{itemize}
		\item[$\mathrm{(i)}$] For any $\lambda
		\in(1,\infty)$ and any cube $Q\subset\rn$, one has
		$\omega(\lambda Q)\leq [\omega]_{A_p
			(\rn)}\lambda^{np}\omega(Q);$
		\item[$\mathrm{(ii)}$]
		$$[\omega]_{A_p(\rn)}=\sup_{Q \subset\rn}
		\sup_{\substack{f\mathbf{1}_Q\in L^p_\omega(\rn)\\
				\int_Q|f(t)|^p\omega(t)\,dt\in(0,\infty)}}\frac
		{[\frac{1}{|Q|}\int_Q|f(t)|\,dt]^p}
		{\frac{1}{\omega(Q)}\int_Q|f(t)|^p\omega(t)\,dt},$$
		where the supremum is taken over all cubes $Q\subset\rn$.
	\end{itemize}
\end{lemma}
The following conclusion shows that $C_{\mathrm{c}}^2(\rn)\subset C_{\mathrm{c}}^\ast(\rn).$
\begin{proposition}\label{1236}
Let $f\in C_{\mathrm{c}}^2(\rn).$ Then,
for any $x\in\rn,$
\begin{align}\label{1235}
\lim_{r\to0}\sup_{|h|<r}\frac{|f(x+h)-f(x)|}{r}=|\nabla f(x)|
\end{align}
converges uniformly on $\rn$.
Moreover, for any $x\in\rn,$ $|\nabla f(x)|=\operatorname{Lip} f(x)=\lip f(x)$ and  $C_{\mathrm{c}}^2(\rn)\subset \cC_{\mathrm{c}}^\ast(\rn).$
\end{proposition}
\begin{proof}
Let all the notation be as in the present proposition. Let $f\in C_{\mathrm{c}}^2(\rn)$
and $r\in(0,\infty).$ By the Taylor remainder theorem, we find that,
for any $x,\ h:=(h_1,\ldots,h_n)\in\rn
$, there exists a $\theta\in(0,1)$ such that
\begin{align*}
f(x+h)=f(x)+\sum_{i=1}^n\frac{\p f(x)}{\p x_i}h_i+\sum_{i,j=1}^n\frac{\p^2 f(x+\theta h)}{2\p x_i \p x_j}h_ih_j.
\end{align*}
From this, we deduce that, for any $x\in\rn$ and $h\in\rn$ with $|h|<r,$
\begin{align*}
\left|\sum_{i=1}^n\frac{\p f(x)}{\p x_i}\frac{h_i}{r}\right|-\frac{Cn^2r}{2}\leq\frac{|f(x+h)-f(x)|}{r}\leq |\nabla f(x)|+\frac{Cn^2r}{2},
\end{align*}
where
$$C:=\sup_{x\in\rn}\sum_{i,j=1}^n\left|\frac{\p^2 f(x)}{\p x_i \p x_j}\right|.$$
This implies that
\begin{align*}
\left|\sup_{|h|<r}\frac{|f(x+h)-f(x)|}{r}-|\nabla f(x)|\right|\leq \frac{Cn^2r}{2}.
\end{align*}
Using this, we conclude that \eqref{1235} holds true, which completes the proof of Proposition  \ref{1236}.
\end{proof}
As an application of Theorem \ref{thm-2-3}, we have the following conclusion.
\begin{theorem}\label{corollary9999}
	Let $ p\in[1,\infty)$ and $\og\in A_p(\rn)$. Then, for any $f\in
	C_{\mathrm{c}}^2(\rn)$
    \begin{align*}
    &\sup_{\ld\in(0,\infty)}\ld^p\int_{\rn}\int_\rn\mathbf{1}_{\{(x,y)\in\rn\times\rn:\ |f(x)-f(y)|>\ld|x-y|[\omega(B(x,|x-y|))]^{1/p}\}}(x,y)\omega(x)\omega(y)\,dx\,dy\\
    &\quad\sim\int_{\rn}|\nabla f(x)|^p\omega(x)\,dx,
    \end{align*}
	where the positive
	equivalence constants  depend only on $n,$ $p,$ and $[\omega]_{A_p(\rn)}$.
\end{theorem}
\begin{proof}
	Let all the notation be as in the present theorem. Let $\omega\in A_p(\rn)$.
	Using Theorem \ref{thm-2-3} with $\cx,$ $\rho(x,y),$ and $d\mu(x)$
	replaced, respectively, by $\rn$, $|x-y|,$ and
	$\omega(x)dx,$
	to show Theorem \ref{corollary9999},
	we only need to prove that, for any $f\in C_{\mathrm{c}}^2(\rn)$, $f$
	satisfies a $(p,p)$-Poincar\'e inequality with
	constants $\tau=1,$ $C_1,$ and $C_2$ independent of $f.$
	For any ball $B\subset\rn$ and $\varphi\in
	C_{\mathrm{b}}(\rn),$ let
	$$\ell_{B}(\varphi):=\frac{1}{|B|}\int_{B}\varphi(x)\,dx.$$
From Proposition \ref{1236}, Lemma \ref{Lemma2.1}(ii),  and
\cite[Corollary 1.8]{PR20},
it
follows that there exist some  constants
$C_1,$ $C_2\in [1,\infty)$ such that, for any
$f\in C_{\mathrm{c}}^2(\rn)$ and $\vi\in C_{\mathrm{b}}(\rn),$
\eqref{3-9-a}  and \eqref{2-8-eq} hold
true with $q=p$ and $\tau=1.$
Thus, for any $f\in C_{\mathrm{c}}^2(\rn)$, $f$
satisfies a $(p,p)$-Poincar\'e inequality with
constants $\tau=1,$ $C_1,$ and $C_2$ independent of $f.$ This finishes
	the proof of Theorem \ref{corollary9999}.
\end{proof}
\begin{remark}
	Let $\omega=1$. In this case, Theorem \ref{corollary9999} is just \cite[Theorem 1.1]{ref8}.
\end{remark}
As the applications of Corollaries \ref{corollart6.5} and \ref{corollary6.6}, we
have the following conclusions. Since the proof  is similar to that of
Theorem \ref{corollary9999}, we omit the details here.
\begin{corollary}\label{10}
	Let $ p\in[1,\infty)$ and $\og\in A_p(\rn)$. Then there exists a positive constant $C$ such that, for any $f\in
	C_{\mathrm{c}}^2(\rn),$
	\begin{align*}
	&\sup_{\ld\in(0,\infty)}\ld\left[\int_{\rn}\int_\rn\mathbf{1}_{\{(x,y)\in\rn\times\rn:\ |f(x)-f(y)|>\ld|x-y|^{1/p}[\omega(B(x,|x-y|))]^{1/p}\}}(x,y)\omega(x)\omega(y)\,dx\,dy\right]^{\frac{1}{p}}\\
	&\quad \leq C\|f\|^{1-\f1p}_{L^\infty(\rn,\omega\,dx
		)}\left[\int_{\rn}|\nabla f(x)|\omega(x)\,dx\right]^{\f1p}.
	\end{align*}
\end{corollary}
\begin{remark}
Let $\omega=1$. In this case, Corollary \ref{10} is just \cite[Corollary 5.1]{ref8}.
\end{remark}
\begin{corollary}\label{11}
	Let $s_1 \in(0,1)$, $p_1\in(1,\infty),$ and
	$\ta\in (0,1)$. Let $s\in (s_1, 1)$ and
	$p\in (1, p_1)$ satisfy $s:=(1-\theta) s_1+\theta$
	and $ \frac{1}{p}:=\frac{1-\theta}{p_1}+\theta$.
	Let  $\og\in A_p(\rn)$. Then there exists a positive constant $C$ such that,  for any $f\in
	C_{\mathrm{c}}^2(\rn)$,
	\begin{align*}
		&\sup_{\ld\in(0,\infty)}\ld\left[
		\int_\rn\int_\rn\mathbf{1}_{\{(x,y)\in\rn\times\rn:\ |f(x)-f(y)|>\ld|x-y|^{s}[\omega(B(x,|x-y|))]^{1/p}\}}(x,y)\omega(x)\omega(y)\,dx\,dy\right]^{\f1p}\\\noz
		&\quad\leq C\left[\int_{\rn}|\nabla f(x)|\omega(x)\,dx\right]^{\theta}\left[\int_{\rn}
		 \int_{\rn}\frac{|f(x)-f(y)|^{p_1}}{|x-y|^{s_1p_1}\omega(B(x,|x-y|))}
\omega(x)\omega(y)\,dx\,dy\right]^{\frac{1-\theta}{p_1}}.
	\end{align*}
\end{corollary}
\begin{remark}
Let $\omega=1$. In this case, Corollary \ref{11} is just \cite[Corollary 5.2]{ref8}.
\end{remark}
\subsection{Complete $n$-dimensional Riemannian Manifolds with Non-negative Ricci Curvature}\label{1r3}
Let us recall the notion about Riemannian manifolds,
which can be found in \cite[Chapters 1 and 2]{c92}.
\begin{definition}\label{dayi}
A Hausdorff space $M$ is called a \emph{smooth manifold of dimension $n$}
if there exist a family $\{U_\al\}_{\al\in\Gamma}$
of open sets in $\rn$, and
a family of homeomorphisms,
$$\{x_\al:\ U_\al\to x_\al(U_\al)\subset M\}_{\al\in\Gamma}$$
with $x_\al(U_\al):=\{p\in M:\ p=x_\al(z),\,z\in U_\al\}$,
such that
\begin{itemize}
\item [$\mathrm{(i)}$] $\bigcup_{\al\in\Gamma}x_\al(U_\al)=M;$
\item [$\mathrm{(ii)}$] for any  $\al,\ \beta\in\Gamma$ with $W:=x_\al(U_\al)\cap x_\beta(U_\beta)\neq \emptyset$,
    the mappings $x_{\beta}^{-1}\circ x_\al$ are infinitely differentiable in $W$;
\item [$\mathrm{(iii)}$] the family $\{(U_\al,x_\al)\}_{\al\in\Gamma}$ satisfy
that, if there exist an open set $V\subset\rn$
and a
homeomorphism $y:\ V\to y(V)\subset M$
such that $\{(U_\al,x_\al)\}_{\al\in\Gamma}\cup\{(V,y)\}$ satisfy (ii), then  $(V,y)\in\{(U_\al,x_\al)\}_{\al\in\Gamma}$.
\end{itemize}
\end{definition}
The pair $(U_\al,x_\al)$ in Definition \ref{dayi} with $p\in x_\al(U_\al)$ is called a \emph{parametrization of M at p}.
\begin{definition}
Let $M_1$ and $M_2$ be, respectively,  $n$-dimensional and  $m$-dimensional smooth manifolds. A mapping $\varphi:\ M_1\to M_2$ is
said to be \emph{smooth} at $p\in M_1$ if, for any given  parametrization $(V,y)$ of $M_2$ at $\varphi(p)$, there exists  a parametrization $(U,x)$ of $M_1$ at $p$ such that $\varphi(x(U))\subset y(V)$ and the mapping
$$y^{-1}\circ\varphi\circ x:\ U\subset\rn\to\mathbb{R}^m$$
is infinitely differentiable at $x^{-1}(p).$
A mapping $\varphi$ is called a \emph{smooth mapping} if $\varphi$ is smooth at any point $p\in M_1.$
\end{definition}
Let $M$ be a smooth manifold as in Definition \ref{dayi}. A \emph{smooth curve} $\al$ in $M$ is a smooth mapping $\al:\ (-\epsilon,\epsilon)\subset\mathbb{R}\to M$. A \emph{smooth function} $f$ in $M$ is
a smooth mapping $f:\ M\to\mathbb{R}.$
Denote by  the  \emph{symbol} $C_{\mathrm{c}}^\infty(M)$ the  set of all smooth functions on $M$
with compact support.
\begin{definition}
Let $M$ be a smooth manifold and
$\al: (-\epsilon,\epsilon)\subset\mathbb{R}\to M$ be a smooth curve. Let $\al(0)=p$ and $\mathcal{D}$ be the set of all real-valued functions $f$ on $M$ such that $f$ is smooth at $p.$ The \emph{tangent vector} to the curve $\al$
at $t=0$ is a mapping $\al'(0):\ \mathcal{D}\to\mathbb{R}$ defined by setting that, for any $f\in\mathcal{D},$
$$\al'(0)(f):=\left.\frac{d(f\circ\al)}{dt}\right|_{t=0}.$$
A \emph{tangent vector} at $p$ is a tangent vector at $t=0$ to some curve $\al:(-\epsilon,\epsilon)\subset\mathbb{R}\to M$ with $\al(0)=p.$ The set of all tangent vectors  at $p$ is denoted by $T_pM.$
\end{definition}
Let $M$ be a $n$-dimensional smooth manifold and $p\in M.$ Suppose that $(U,x)$ is a parametrization of $M$ at $p$ and $x(h_1,\ldots,h_n)=p,$ where $(h_1,\ldots,h_n)\in U.$ For any $i\in\{1,\ldots,n\}$ and $t$ sufficiently small,  let $\al_i(t):=x(h_1,\ldots,h_i+t,\ldots,h_n).$
Then $\al_i$ is a smooth curve with $\al_i(0)=p$ for any  $i\in\{1,\ldots,n\}$.
Denote the tangent vector $\al_i'(0)$ by $\left.\frac{\p}{\p x_i}\right|_p$ (see, for instance, \cite[p.\,8]{c92}).
  A \emph{vector field} $X$ on  $M$ is a map on $M$  such that, for any  $p\in M,$  $X(p)\in T_pM.$ For any $p\in M,$ noticing that $\{\frac{\p}{\p x_1}|_p,\ldots,\frac{\p}{\p x_n}|_p\}$ is a basis of $T_pM$ (see, for instance, \cite[p.\,8]{c92}), we can write
\begin{align}\label{vg}
X(p)=\sum_{i=1}^na_i(p)\left.\frac{\p}{\p x_i}\right|_p,
\end{align}
where $a_i(p)\in\mathbb{R}$ for any $i\in\{1,\ldots,n\}.$
\begin{definition}
Let $M$ be a $n$-dimensional smooth manifold as in Definition \ref{dayi}. A  vector field $X$ on  $M$ is called a \emph{smooth vector field} if, for any $p\in M$ and any parametrization $(U,x)$ of $M$ at $p$, the coefficient function $a_i(\cdot)$ as in \eqref{vg} is a smooth function on $U$ for any $i\in\{1,\ldots, n\}.$

\end{definition}
\begin{definition}\label{eee}
	A \emph{Riemannian metric} on a $n$-dimensional smooth manifold $M$ is a map  on $M$ such that, for any  $p\in M,$  $\langle\cdot,\cdot\rangle_p$ is an inner product on the tangent space $T_pM,$
which varies smoothly in the following sense: if $(U,x)$ is a parametrization of $p,$
then, for any $i,j\in\{1,\ldots,n\},$ $$\left\langle \left.\frac{\p}{\p x_i}\right|_{(\cdot)},\left.\frac{\p}{\p x_j}\right|_{(\cdot)}\right\rangle_{(\cdot)}$$ is a smooth function on $U.$
	A smooth manifold with a Riemannian metric is called a \emph{Riemannian manifold}.
\end{definition}
Let $M$ be a $n$-dimensional Riemannian manifold. For any $p\in M$ and $v\in T_pM,$ let
\begin{align*}
|v|_M:=\sqrt{\langle v,v\rangle_p},
\end{align*}
where $\langle \cdot,\cdot\rangle_p$ is as in Definition \ref{eee}.
\begin{definition}\label{ymy}
Let $M$ be a  Riemannian manifold as in Definition \ref{eee} and $f$ a smooth function on $M.$ The \emph{gradient of}  $f$, denoted by $\nabla f$, is a smooth vector field such that, for any  smooth vector field $X$ and any $p\in M$,
\begin{align}\label{tidu}
\langle\nabla f(p),X(p)\rangle_p=X(p)(f).
\end{align}
\end{definition}
\begin{remark}
Let $M$ be a  Riemannian manifold as in Definition \ref{eee}. By \cite[p.\,343]{L13}, we find that, for any smooth function $f$ on $M$, there exists a unique  smooth vector field satisfying \eqref{tidu}. This implies that, for any smooth function $f$ on $M$, $\nabla f$  exists and is unique. Thus, for any smooth function $f$ on $M$, $\nabla f$ in Definition \ref{ymy} is well defined.
\end{remark}

Let $M$ be a  Riemannian manifold as in Definition \ref{eee}. Let $\rho_M$ be the canonical distance function  on $M$ (see  \cite[p.\,146]{c92} for the  precise definition). By \cite[Lemma 6.2]{L13}, we find that $p_M$ is a metric on $M.$ For any $x\in M$ and $r\in(0,\infty),$ let $$B_M(x,r):=\{y\in M:\ \rho_M(x,y)<r\}.$$
Similarly to Proposition \ref{1236}, we have  the following conclusion.
\begin{proposition}\label{944}
	Let $M$ be a $n$-dimensional Riemannian manifold and $f\in C_{\mathrm{c}}^\infty(M).$ Then,
	for any $x\in M,$
	\begin{align}\label{1066}
		\lim_{r\to0}\sup_{y\in B_M(x,r)}\frac{|f(x)-f(y)|}{r}=|\nabla f(x)|_M
\end{align}
converges uniformly on $M$.	Moreover, for any $x\in M,$ $|\nabla f(x)|=\operatorname{Lip} f(x)=\lip f(x)$ and  $C_{\mathrm{c}}^\infty(M)\subset \cC_{\mathrm{c}}^\ast(M).$
\end{proposition}
\begin{proof}
Let all the notation be as in the present proposition. Let $p\in M$ and $\{\nu_i\}_{i=1}^n$
be a sequence of  smooth vector fields in a neighborhood $U$ of $p$ such that, for any
$x\in U,$ $\{\nu_i(x)\}_{i=1}^n$ is an orthonormality basis of $T_xM$. For any $r\in(0,\infty),$  let $$B(\mathbf{0},r):=\{y\in\rn:\ |y|<r\}.$$    By \cite[p.\,72]{c92}, we find that
there exists a positive constant $\delta(p)$ such that $B_M(p,\delta(p))\subset U$ and,
for any $(x,t)\in B_M(p,\delta(p))\times B(\mathbf{0},\delta(p)),$
\begin{align*}
F(x,t):=F(x,t_1,\ldots,t_n):=f\circ\exp_x\left(\sum_{i=1}^nt_i\nu_i(x)\right)
\end{align*}
is  smooth and
\begin{align*}
C:=\sup_{\substack{x\in B_M(p,\delta(p))\\t\in B(\mathbf{0},\delta(p))}}\sum_{i,j=1}^n\left|\frac{\p^2 [F(x,(\cdot))]}{\p t_i\p t_j}(t)\right|<\infty,
\end{align*}
where $\exp_x$ is the exponential map (see \cite[p.\,65]{c92} for the precise definition).
By the definition of $\nabla$ and \cite[Proposition 2.7 of Chapter 3]{c92}, we find that, for any $i\in\{1,\ldots,n\}$ and $x\in B_M(p,\delta(p))$,
\begin{align*}
\langle \nabla f(x),\nu_i(x)\rangle_x=\nu_i(x)(f)
=\left.\frac{df (\exp_x((\cdot)\nu_i(x)))}{d t}\right|_{t=0}=
\frac{\p [F(x,\cdot)]}{\p t_i}(\mathbf{0}).
\end{align*}
This, together with the fact that $\{\nu_i(x)\}_{i=1}^n$
is an orthonormality basis of $T_xM$, implies that
\begin{align*}
\nabla f(x)=\sum_{i=1}^n\frac{\p [F(x,\cdot)]}{\p t_i}(\mathbf{0})\nu_i(x)
\end{align*}
and
\begin{align*}
|\nabla f(x)|_M&=\sqrt{\langle \nabla f(x), \nabla f(x)\rangle_x}=
\sqrt{\sum_{i,j=1}^n\frac{\p [F(x,\cdot)]}{\p t_i}(\mathbf{0})\frac{\p [F(x,\cdot)]}{\p t_j}(\mathbf{0})\langle\nu_i(x),\nu_j(x)\rangle_x}\\
&=\sqrt{\sum_{i=1}^n\left|\frac{\p [F(x,\cdot)]}{\p t_i}(\mathbf{0})\right|^2}=:|\nabla F(x,\mathbf{0})|.
\end{align*}
On the other hand, by the Taylor remainder theorem, we conclude that, for any given $(x,t)\in B_M(p,\delta(p)/2)\times B(\mathbf{0},\delta(p))$, there exists a $\theta\in(0,1)$ such that
\begin{align}\label{1029}
F(x,t)=F(x,t_1,\ldots,t_n)=f(x)+\sum_{i=1}^n\frac{\p [F(x,\cdot)]}{\p t_i}(\mathbf{0})t_i+\sum_{i,j=1}^n\frac{\p^2 [F(x,(\cdot))]}{2\p t_i\p t_j}(\theta t)t_it_j.
\end{align}
Let $r\in(0,\delta(p)/2)$ be sufficiently small and $y\in B_M(x,r)\subset
 B_M(p,\delta(p)).$ By \cite[Corollary 6.11]{L13}, we find that the map $\exp_x$ is a bijection from $\{v\in T_xM:\ |v|_M<r\}$ to $B_M(x,r)$. From this and the assumption that $\{\nu_i(x)\}_{i=1}^n$
is a orthonormality basis of $T_xM$, it follows that there exists an $(a_1(y),\,\ldots,a_n(y))\in
 B(\mathbf{0},r)$ such that
\begin{align*}
y=\exp_x\left(\sum_{i=1}^na_i(y)\nu_i(x)\right).
\end{align*}
Using this and \eqref{1029}, we conclude that
\begin{align*}
f(y)-f(x)&=F(x,a_1(y),\ldots,a_n(y))-f(x)\\\noz
&=\sum_{i=1}^n\frac{\p [F(x,\cdot)]}{\p t_i}(\mathbf{0})a_i(y)+\sum_{i,j=1}^n\frac{\p^2 [F(x,(\cdot))]}{2\p t_i\p t_j}(\theta a_1(y),\ldots,\theta a_1(y))a_i(y)a_j(y).
\end{align*}
This, together with the fact that the map $y\to(a_1(y),\ldots,a_n(y))$ is a bijection from $B_M(x,r)$ to $B(\mathbf{0},r)$, implies  that
\begin{align*}
|\nabla F(x,\mathbf{0})|-\frac{Cn^2r}{2}\leq \sup_{y\in B_M(x,r)}\frac{|f(y)-f(x)|}{r}\leq |\nabla F(x,\mathbf{0})|+\frac{Cn^2r}{2}.
\end{align*}
 Thus, for any $x\in B_M(p,\delta(p)/2)$,
\begin{align*}
		\lim_{r\to0}\sup_{y\in B_M(x,r)}\frac{|f(x)-f(y)|}{r}=|\nabla f(x)|_M
\end{align*}
converges uniformly on $B_M(p,\delta(p)/2)$.
From this and the  boundedness of $\supp f$, we deduce that \eqref{1066}
holds true. This finishes the proof of Proposition \ref{944}.
\end{proof}

Now, let $M$ be a complete $n$-dimensional Riemannian manifold whose Ricci
curvature is non-negative (see  \cite[p.\,97 and Charpter 7]{c92} for the precise definition).
Denote the canonical measure on $M$ by $\mu$ (see \cite[p.\,44]{c92} for the precise definition).
By \cite[Section 10.1]{hp00}, we find that $\mu$
is a doubling measure on $M$ and there exists a positive constant $C$ such that, for any
$f\in C_{\mathrm{c}}^\infty(M)$  and any ball
$B_M:=B_M(y_{B_M},r_{B_M})=\{z\in M:\ \rho_M(z,y_{B_M})<r_{B_M}\}\subset M,$
\begin{align*}
\int_{{B_M}}|f(x)-f_{B_M}|\,d\mu(x)\leq Cr_{{B_M}}\int_{{B_M}}|\nabla f(x)|\,dx,
\end{align*}
where $$f_{B_M}:=\frac{1}{\mu({B_M})}
\int_{B_M}f(x)\,d\mu(x).$$ This, together with
Proposition \ref{944}, implies that, for any  $f\in C_{\mathrm{c}}^\infty(M),$  $f$ satisfies a
$(1,1)$-Poincar\'e inequality with constants independent of $f$. By
Theorem \ref{thm-2-3} and Corollaries \ref{corollart6.5} and \ref{corollary6.6}, we have the following conclusion, we omit the details here.
\begin{theorem}\label{13}
Let $(M,\,\rho_M,\,\mu)$ be a complete $n$-dimensional Riemannian manifold whose Ricci curvature is non-negative, where $\mu$ is the canonical measure on $M$. Then, for any $f\in C_{\mathrm{c}}^\infty(M)$,
\begin{equation*}
	\sup_{\ld\in(0,\infty)} 	\ld \int_{{M}}\int_{{M}} \mathbf{1}_{{D_{\ld,\,1}}}(x, y) \,d\mu(x) \,d\mu(y) \sim 	\int_{{M}} |\nabla f(x)|_M  \,d\mu(x),
\end{equation*}
where
\begin{align*}
	D_{\ld,\,1} :=\left\{ (x, y)\in {M}\times {M}:\  |f(x)-f(y)|>\ld \rho_M(x,y) V(x,y) \right\}
\end{align*}
and the  positive
equivalence constants are independent of $f$.
\end{theorem}
\begin{remark}
	To the best of our knowledge, the result of Theorem \ref{13} is totally new.
\end{remark}
\begin{corollary}\label{14}
Let $(M,\,\rho_M,\,\mu)$ be a complete $n$-dimensional Riemannian manifold
whose Ricci curvature is non-negative and $p\in[1,\infty)$, where $\mu$ is
the canonical measure on $M$. Then, for any $f\in C_{\mathrm{c}}^\infty(M)$,
\eqref{hehe1} holds trues, where the  positive
equivalence constants are independent of $f$.
\end{corollary}

\begin{remark}
	To the best of our knowledge, the result of Corollary \ref{14} is totally new.
\end{remark}

\begin{corollary}\label{15}
Let $s_1 \in(0,1)$, $p_1\in(1,\infty),$
$\ta\in (0,1)$, and $q\in[1,\infty)$. Let $s\in (s_1, 1)$ and
$p\in (1, p_1)$ satisfy $s:=(1-\theta) s_1+\theta$
and $ \frac{1}{p}:=\frac{1-\theta}{p_1}+\theta$. Let $(M,\,\rho_M,\,\mu)$ be
a complete $n$-dimensional Riemannian manifold whose Ricci curvature is
non-negative and $p\in[1,\infty)$, where $\mu$ is the canonical measure on $M$.
Then, for any $f\in C_{\mathrm{c}}^\infty(M)$, \eqref{hehe2} holds trues, where the  positive
equivalence constants are independent of $f$.
\end{corollary}
\begin{remark}
	To the best of our knowledge, the result of Corollary \ref{15} is totally new.

\end{remark}

\bigskip

\smallskip

\noindent Feng Dai

\smallskip

\noindent Department of Mathemtical and
Statistical Sciences, University of Alberta
Edmonton, Alberta T6G 2G1, Canada

\smallskip

\noindent {\it E-mail}: \texttt{fdai@ualberta.ca}

\bigskip

\noindent Xiaosheng Lin, Dachun Yang
(Corresponding author),
Wen Yuan and Yangyang Zhang

\smallskip

\noindent Laboratory of Mathematics and Complex Systems
(Ministry of Education of China),
School of Mathematical Sciences, Beijing Normal University,
Beijing 100875, People's Republic of China

\smallskip

\noindent {\it E-mails}: \texttt{xiaoslin@mail.bnu.edu.cn} (X. Lin)

\noindent\phantom{{\it E-mails:}} \texttt{dcyang@bnu.edu.cn} (D. Yang)

\noindent\phantom{{\it E-mails:}} \texttt{wenyuan@bnu.edu.cn} (W. Yuan)

\noindent\phantom{{\it E-mails:}} \texttt{yangyzhang@mail.bnu.edu.cn} (Y. Zhang)


\begin{thebibliography}{99}
\bibitem{bbm}J. Bourgain, H. Brezis and
P. Mironescu, Lifting in Sobolev spaces,
J. Anal. Math. 80 (2000),
37-86.

\vspace{-0.3cm}

\bibitem{bbm01}J. Bourgain, H. Brezis and P.
Mironescu, Another look at Sobolev spaces,
Optimal Control and Partial Differential
Equations, IOS, Amsterdam, 2001, 439-455.
	
\vspace{-0.3cm}
	
\bibitem{bbm02}J. Bourgain, H. Brezis and
P. Mironescu, Limiting embedding theorems
for $W^{s,p}$ when $s\uparrow1$ and
applications, J. Anal. Math. 87 (2002), 77-101.
	
\vspace{-0.3cm}

\bibitem{BSYar}D. Brazke, A. Schikorra and P.-L. Yung, Bourgain--Brezis--Mironescu convergence via Triebel--Lizorkin spaces, arXiv: 2109.04159.

\vspace{-0.3cm}

	
\bibitem{ref7}
H. Brezis, How to recognize constant functions.
A connection with Sobolev spaces, Russian
Math. Surveys 57 (2002), 693-708.
	
\vspace{-0.3cm}
	
\bibitem{BSSY21}
H, Brezis, A. Seeger, J. Van Schaftingen and P.-L. Yung,
Families of functionals representing Sobolev norms, arXiv: 2109.02930.


\vspace{-0.3cm}


\bibitem{ref8} H. Brezis, J. Van Schaftingen
and P.-L. Yung, A surprising formula
for Sobolev norms, Proc. Natl. Acad.
Sci. USA118(2021), e2025254118, 6 pp.
	
\vspace{-0.3cm}

\bibitem{BSY212}
H. Brezis, J. Van Schaftingen and P.-L. Yung,
Going to Lorentz when fractional Sobolev, Gagliardo and Nirenberg estimates fail,
Calc. Var. Partial Differential
Equations 60 (2021),  Paper No. 129, 12 pp.

\vspace{-0.3cm}

\bibitem{BD20}T. A. Bui and X.-T. Duong, Sharp weighted estimates for square functions associated to
operators on spaces of homogeneous type, J. Geom. Anal. 30 (2020), 874-900.

\vspace{-0.3cm}

\bibitem{BDL}T. A. Bui, X.-T. Duong and F. K. Ly, Maximal function characterizations for new local
Hardy-type spaces on spaces of homogeneous type, Trans. Amer. Math. Soc. 370 (2018),
7229-7292.

\vspace{-0.3cm}

\bibitem{crs10}
L. Caffarelli, J. M. Roquejoffre and
O. Savin, Non-local minimal surfaces,
Comm. Pure Appl. Math. 63 (2010), 1111-1144.
	
\vspace{-0.3cm}
	
\bibitem{cv11}
L. Caffarelli and E. Valdinoci, Uniform
estimates and limiting arguments for
nonlocal minimal surfaces, Calc. Var.
Partial Differential Equations 41 (2011), 203-240.
	


\vspace{-0.3cm}

\bibitem{CW71}R. R. Coifman and G. Weiss, Analyse Harmonique Non-Commutative sur Certains Espaces
Homog\`enes, (French) \'Etude de Certaines Int\'egrales Singuli\`eres, Lecture Notes in Math. 242,
Springer-Verlag, Berlin-New York, 1971.


\vspace{-0.3cm}

\bibitem{CW77}R. R. Coifman and G. Weiss, Extensions of Hardy spaces and their use in analysis, Bull.
Amer. Math. Soc. 83 (1977), 569-645.

\vspace{-0.3cm}

\bibitem{DLYYZ21}
F. Dai, X. Lin, D. Yang, W. Yuan and
Y. Zhang, Generalization in ball Banach function spaces of Brezis--Van Schaftingen--Yung formulae with applications to fractional Sobolev and Gagliardo--Nirenberg inequalities,
arXiv: 2109.04638.

\vspace{-0.3cm}

\bibitem{c92}
M. P. do Carmo, Riemannian Geometry,
Mathematics: Theory and Applications, Birkh\"auser Boston, Inc., Boston, MA,  1992.

\vspace{-0.3cm}

		
\bibitem{ref1}
L. Grafakos, Classical Fourier Analysis,
third edition, Grad. Texts in Math.
249, Springer, New York, 2014.
	
\vspace{-0.3cm}

\bibitem{GY21}Q. Gu and P.-L. Yung, A new formula for the $L^p$ norm, J. Funct. Anal. 281 (2021), Paper No. 109075, 19 pp.

\vspace{-0.3cm}

\bibitem{hp00}
P. Haj\l asz and P. Koskela, Sobolev Met Poincar\'e, Mem. Amer. Math. Soc.  145 (2000),
 no. 688, x+101 pp.

\vspace{-0.3cm}

\bibitem{HHL16} Ya. Han, Yo. Han and J. Li, Criterion of the boundedness of singular integrals on spaces of
homogeneous type, J. Funct. Anal. 271 (2016), 3423-3464.

\vspace{-0.3cm}

\bibitem{HHL17}Ya. Han, Yo. Han and J. Li, Geometry and Hardy spaces on spaces of homogeneous type in
the sense of Coifman and Weiss, Sci. China Math. 60 (2017), 2199-2218.

\vspace{-0.3cm}

\bibitem{H01}J. Heinone, Lectures on Analysis on Metric Spaces,
Universitext, Springer-Verlag, New York,  2001.

\vspace{-0.3cm}

\bibitem{H2010}T. Hyt\"onen, A framework for non-homogeneous analysis on metric spaces, and the
RBMO space of Tolsa, Publ. Mat. 54 (2010), 485-504.
	
\vspace{-0.3cm}
	
\bibitem{L10}J. Li, Atomic decomposition of weighted Triebel--Lizorkin spaces on spaces of homogeneous
type, J. Aust. Math. Soc. 89 (2010), 255-275.


\vspace{-0.3cm}

\bibitem{LW13}J. Li and L. A. Ward, Singular integrals on Carleson measure spaces $\mathrm{CMO}^p$
on product
spaces of homogeneous type, Proc. Amer. Math. Soc. 141 (2013), 2767-2782.

\vspace{-0.3cm}

\bibitem{K15} K. A. Kopotun, Polynomial
approximation with doubling weights having
finitely many zeros and singularities,
J. Approx. Theory 198 (2015), 24-62.

\vspace{-0.3cm}

\bibitem{L13} J. M. Lee, Introduction to Smooth Manifolds, second edition, Grad. Texts in Math. 218, Springer, New York, 2013.

\vspace{-0.3cm}
	
\bibitem{MT}
G. Mastroianni and V. Totik, Best
approximation and moduli of smoothness
for doubling weights, J. Approx.
Theory 110 (2001), 180-199

\vspace{-0.3cm}
	
\bibitem{m11}
V. Maz'ya, Sobolev Spaces with Applications
to Elliptic Partial Differential Equations,
second, revised and augmented edition,
Grundlehren der Mathematischen
Wissenschaften  342, Springer, Heidelberg, 2011.
	
\vspace{-0.3cm}
		
\bibitem{mm11}
G. Mingione, Gradient potential estimates,
J. Eur. Math. Soc. 13 (2011), 459-486.
	
\vspace{-0.3cm}

\bibitem{N06}E. Nakai, The Campanato, Morrey and H\"older spaces on spaces of homogeneous type, Studia
Math. 176 (2006), 1-19.

\vspace{-0.3cm}

\bibitem{NY97}E. Nakai and K. Yabuta, Pointwise multipliers for functions of weighted bounded mean oscillation on spaces of homogeneous type, Math. Japon. 46 (1997), 15–28.

\vspace{-0.3cm}
		
\bibitem{npv}
E. D. Nezza, G. Palatucci and E. Valdinoci,
Hitchhiker's guide to the fractional Sobolev
 spaces, Bull. Sci. Math.
136 (2012), 521-573.
	
\vspace{-0.3cm}
	
\bibitem{PR20}
C. P\'erez and E. Rela, Degenerate
Poincar\'e--Sobolev inequalities, Trans. Amer.
Math. Soc. 372 (2019), 6087-6133.

\vspace{-0.3cm}

\bibitem{Par}A. Poliakovsky, Some remarks on a formula for Sobolev norms due
to Brezis, Van Schaftingen and Yung, J. Funct. Anal. 282 (2022), no. 3,
Paper No. 109312, 47 pp.

\vspace{-0.3cm}

\end{thebibliography}
\end{document}